\documentclass[12pt, a4paper]{amsart}
\usepackage{enumerate, longtable}
\usepackage{amsmath, amscd, amsfonts, amsthm, amssymb, latexsym, comment, stmaryrd, graphicx}
\usepackage{xcolor}
\usepackage[all]{xy}
\usepackage{a4wide}
\usepackage[
        colorlinks,
        backref,
]{hyperref}

\theoremstyle{plain}
\newtheorem{theorem}{Theorem}[section]
\newtheorem{proposition}[theorem]{Proposition}
\newtheorem{lemma}[theorem]{Lemma}

\newtheorem{conjecture}[theorem]{Conjecture}
\newtheorem{corollary}[theorem]{Corollary}

\theoremstyle{definition}

\newtheorem{remark}[theorem]{Remark}

\newtheorem{example}[theorem]{Example}

\numberwithin{equation}{section}

\newcommand{\ZZ}{\mathbb{Z}}

\newcommand{\FF}{\mathbb{F}}

\newcommand{\ts}{\chi}

\def\FF{\mathbb{F}}

\begin{document}

\title[Correlations in function fields]{Correlations of sums of two squares and other arithmetic functions in function fields}

\author{Lior Bary-Soroker}
\address{Raymond and Beverly Sackler School of Mathematical Sciences, Tel Aviv University, Tel Aviv 69978, Israel}
\email{barylior@post.tau.ac.il}
\author{Arno Fehm}
\address{School of Mathematics, Alan Turing Building, The University of Manchester, Oxford Road, Manchester, M13 9PL, United Kingdom}
\email{arno.fehm@manchester.ac.uk}

\maketitle

\begin{abstract}
We investigate a function field analogue of a recent conjecture on autocorrelations of sums of two squares
by Freiberg, Kurlberg and Rosenzweig, which generalizes an older conjecture by Connors and Keating.
In particular, we provide extensive numerical evidence and prove it in the large finite field limit.
Our method can also handle correlations of other arithmetic functions and we 
give applications to (function field analogues of) the average of sums of two squares on shifted primes,
and to autocorrelations of higher divisor functions twisted by a quadratic character.
%
%
\end{abstract}


\section{Introduction}

\noindent
We study function field analogues of conjectures on autocorrelations of sums of two squares.

\subsection{Correlations of arithmetic functions}
A basic statistical property of an arithmetic function $\psi$ is its \emph{mean value}; that is, the asymptotic as $x\to \infty$ of 
\begin{equation}\label{eq:MV}
\left<\psi(n)\right>_{n\leq x} = \frac{1}{x} \sum_{n\leq x} \psi(n).
\end{equation}
More information is given by the \emph{cross-correlations} of arithmetic functions $\psi_1,\ldots, \psi_k$, 
which are defined at $(h_1,\ldots, h_k)\in \ZZ^k$ as the asymptotic as $x\to \infty$ of 
\begin{equation}\label{Cor}
\left<\prod_{i=1}^{k} \psi_i(n+h_i)\right>_{n\leq x}=\frac{1}{x}\sum_{n\leq x}\psi_1(n+h_1) \cdots \psi_k(n+h_k).
\end{equation}
In the special case when all the $\psi_i$ are equal, we use the term \emph{autocorrelations}. 
Some of the most famous theorems and problems in number theory are about these statistical properties for certain specific arithmetic functions,
like the Hardy-Littlewood prime tuple conjecture, a quantitative version of the twin prime conjecture,
which can be expressed in terms of the autocorrelations of the von Mangoldt function $\Lambda$.

\subsection{Autocorrelation of sums of two squares}
An integer $n$ is  a sum of two squares if there exist  $x,y\in \ZZ$ such that $n=x^2+y^2$;
i.e., it is a norm of the Gaussian integer $x+iy\in \ZZ[i]$. We let 
\begin{equation}
b(n) = \begin{cases}
1, & \mbox{if there exist }x,y\in \mathbb{Z}\mbox{ such that } n = x^2+y^2, \\
0, & \mbox{otherwise.}
\end{cases}
\end{equation}
The study of the statistics of $b(n)$ has a long history:  
Already Landau \cite{Landau} gives the mean value of $b$ as
\begin{equation}\label{LandauThm}
\left<b(n)\right>_{n\leq x} \;=\; \frac{1}{x} \sum_{n\leq x} b(n) \;\sim\; K \cdot\frac{1}{\sqrt{\log x}}, \qquad x\to \infty,
\end{equation}
where 
\begin{equation}
K\;=\;\frac{1}{\sqrt{2}}\prod_{p\, \equiv\, 3\,(\mathrm{mod}\,4)}(1-p^{-2})^{-1/2}\;\approx\; 0.764  
\end{equation}
is the Landau-Ramanujan constant. 
The distribution of sums of two squares was then studied
intensively through the behavior of $b$ in short intervals 
in many works by various authors, including \cite{Hooley3,Iwaniec,Hooley4}, see the introduction of \cite[\S1.2]{BBF} for a brief history.

As for autocorrelations of $b$, one has lower and upper bounds of the right order of magnitude for
the pair autocorrelations:
\begin{equation}\label{eq:paircorrelationbounds}
\frac{1}{\log x} \ll \left< b(n)b(n+h)\right>_{n\leq x} \ll \frac{1}{\log x}.
\end{equation}
For $h=1$, the upper bound was proved by Rieger 
\cite{Rieger} while the lower bound by Indlekofer and Schwarz in \cite{IndlekoferSchwarz,Indlekofer,Schwarz} (see Kelly \cite{Kelly} and Bantle \cite{Bantle} for short interval versions of the lower bound in \eqref{eq:paircorrelationbounds}). Hooley \cite{Hooley3} proved \eqref{eq:paircorrelationbounds} for general $h\neq 0$. As Hooley \cite{Hooley1} asserts, determining the asymptotics of $\left<b(n)b(n+h)\right>_{n\leq x}$ brings ``\emph{much the same difficulties}'' 
as computing $\left<\Lambda(n)\Lambda(n+h)\right>_{n\leq x}$, which is a special case of the aforementioned Hardy-Littlewood conjecture.

For triple autocorrelation, as mentioned by Cochrane and Dressler \cite{CochraneDressler}, it is trivial that there are infinitely many triples $(n-1,n,n+1)$ with $b(n-1)b(n)b(n+1)=1$ (since one triple like this generates another one, namely $(n^2-1,n^2,n^2+1)$) and they give an upper bound of the expected order of magnitude
\[
\left<b(n-1)b(n)b(n+1)\right>_{n\leq x}\ll \frac{1}{(\log x)^{3/2}}.
\] 
Hooley \cite{Hooley2} finds infinitely many $n$ with $b(n)b(n+h_1)b(n+h_2)=1$ for any $h_1,h_2$.
In the study of Connors and Keating \cite{ConnorsKeating} on the two-point correlations in the quantum spectrum of the square
billiard, they give a conjectural pair autocorrelation of $b$ at $h\neq 0$, 
\begin{equation}\label{eq:sos_paircorrelation}
\left< b(n) b(n+h) \right>_{n\leq x} \sim W_h \cdot \frac{1}{\log x}, \qquad x\to \infty
\end{equation}
where $W_h$ is an explicit constant depending on $h$
(cf.~\cite{Iwaniec} for a contradicting conjecture in the case $h=1$), 
and they provide numerical data in support of their conjecture. 
Assuming \eqref{eq:sos_paircorrelation},
Smilansky \cite{Smilansky} calculates the second moment of the distribution of $b$ in short intervals of length $\lambda \sqrt{\log x}$
and shows it is consistent with a Poisson distribution.
Freiberg, Kurlberg, and Rosenzweig \cite{FKR} give heuristics for higher level autocorrelations
and conjecture that for a tuple $h=(h_1,\ldots, h_k)$ of pairwise distinct integers,
\begin{equation}\label{eq:autocorrelation}
\left< b(n+h_1)\cdots b(n+h_k) \right>_{n\leq x}  \sim \mathfrak{S}_{h} \cdot \left<b(n)\right>_{n\leq x}^{k}, \qquad x\to \infty
\end{equation}
with 
\[
\mathfrak{S}_h =\prod_{p} \frac{\delta_h(p)}{\delta_0(p)},
\]
where the product runs over all primes, and 
\[
\delta_h(p) = \lim_{\nu\to \infty} \frac{\# \{n\in \mathbb{Z}/p^{\nu}\mathbb{Z}:\forall i\exists a_i,b_i\;\mbox{s.t. }n+h_i\equiv a_i^2+b_i^2 \mod p^{\nu}\}}{p^{\nu}}.
\] 
It can happen that there is a local obstruction, leading to $\mathfrak{S}_h=0$, e.g.~when the $h_i$ cover all residue classes modulo $4$,
but \cite{FKR} show that $\mathfrak{S}_h>0$ if $k\leq 3$.
We should also note that \eqref{eq:sos_paircorrelation} is the special case $k=2$ of \eqref{eq:autocorrelation}, 
i.e.~$W_h=\mathfrak{S}_{(0,h)}$,
see \cite[Discussion~1.1]{FKR}. 
Assuming \eqref{eq:autocorrelation}, \cite[Theorem 1.4]{FKR} deduces that the distribution of the number of sums of two squares in short intervals of typical length is indeed Poisson. 
The goal of this paper is to provide evidence for \eqref{eq:autocorrelation} by studying this problem in the function field setting. 


\subsection{Correlations in the function field setting}
In this setting, we replace the ring of integers by the ring of polynomials $\FF_q[T]$ over a finite field $\FF_q$ with $q$ elements. 
The positive integers up to $x$ are modeled by the subset $M_{n,q}\subseteq \FF_q[T]$ of monic polynomials of degree $n$ and the \emph{prime polynomials} are the 
monic irreducible polynomials. 
See e.g.~\cite{Rudnick} for the classical analogue of the prime number theorem and a survey of some of the recent work on number theory in function fields.
Our arithmetic functions are complex valued functions $\psi$ on the monic polynomials $M_q=\bigcup_{n=1}^\infty M_{n,q}$.
In a general point of view, our goal is to understand the cross-correlations of arithmetic functions 
$\psi_1,\ldots, \psi_k$ on $M_q$ 
at $(h_1,\dots,h_k)\in\mathbb{F}_q[T]^k$,
\begin{equation}\label{eq:corsiff}
\left<\prod_{i=1}^k\psi_i(f+h_i) \right>_{f\in M_{n,q}}
= \frac{1}{q^n} \sum_{f\in M_{n,q}} \psi_1(f+h_1) \cdots \psi_{k}(f+h_k)
\end{equation}
as the parameter $q^n = \#M_{n,q}$ is large
(and $n>{\rm deg}(h_i)$ for all $i$ to avoid technical difficulties).
This parameter can be large, in particular, either when $n$ is much larger than $q$, which we call the \emph{large degree limit},
or when $q$ is much larger than $n$, which we call the \emph{large finite field limit}.

Typically, in the large degree limit, one knows no more than what is known in number fields assuming the Generalized Riemann Hypothesis\footnote{There are, however, several exceptions to this typical phenomenon; see for example, \cite{Hall,Poonen}.} (which is, of course, a theorem in function fields
). 
In the large finite field limit one can often go much further than what can be done in the number field setting 
or in the large degree limit. 
An extensive study by several authors \cite{ABR,BB,Carmon,CarmonRudnick,Pollack} has led to a complete understanding of \eqref{eq:corsiff} in this limit 
for the family of arithmetic functions depending on \emph{cycle structure} (see \cite[Theorem~1.4]{ABR}).

\subsection{Sums of two squares in the function field setting}
There is a recent series of works on a certain function field analogue of sums of two squares,
which we now recall briefly:
For $f\in M_{q}$ we let
\begin{eqnarray}
\label{eqn:def_bq}
b_q(f) &=& 
\begin{cases}
1,& \mbox{if }f =A^2 + TB^2, A,B\in\mathbb{F}_q[T]\\
0,& \mbox{otherwise.}
\end{cases},
\end{eqnarray}
i.e.~we consider norms from the ring $\FF_q[\sqrt{- T}]$, which we take as the analogue of $\ZZ[i]$. 
With this definition, Smilansky, Wolf and the first named author \cite{BSW} give asymptotics for $\left<b_q(f)\right>_{f\in M_{n,q}}$ in the limits $q\to \infty$ and $n\to \infty$,
and Gorodetsky \cite{Gorodetsky} extends this to 
\begin{eqnarray}
 \left<b_q(f)\right>_{f\in M_{n,q}} &\sim& K_q\cdot\frac{1}{4^n}\left(2n\atop n\right),\quad q^n\rightarrow\infty,
\end{eqnarray}
where
\begin{equation}
\label{eqn:def_Kq}
 K_q=(1-q^{-1})^{-\frac{1}{2}}\prod_{\chi_q(P)=-1}(1-|P|^{-2})^{-\frac{1}{2}}=1+O(q^{-1})
\end{equation}
is an explicit constant depending only on $q$ (see Section \ref{sec:prelim} for notation).
Moreover, Bank and the two authors \cite{BBF} determine the mean value of $b$ in short intervals in the limit $q\rightarrow\infty$.
In the spirit of these works, we can formulate a function field analogue of \eqref{eq:autocorrelation}
for autocorrelations of $b_q$:

\begin{conjecture}\label{conj}
Fix $N\geq 1$ and $k\geq 1$.
Then for $q$ an odd prime power, 
$n\geq N$,
and $h_1,\dots,h_k\in\mathbb{F}_q[T]$ of degree less than $N$ and pairwise distinct,
\begin{eqnarray*}
\left<\prod_{i=1}^k b_q(f+h_i)\right>_{f\in M_{n,q}} &\sim& \mathfrak{S}_{q,h}\cdot \left<b_q(f)\right>_{f\in M_{n,q}}^k \\
  &\sim& \mathfrak{S}_{q,h}\cdot K_q^k\cdot\frac{1}{4^{nk}}\left(2n\atop n\right)^k
\end{eqnarray*}
uniformly as $q^n\rightarrow\infty$, where $K_q$ is defined as in \eqref{eqn:def_Kq} and
\begin{eqnarray}
\label{eqn:def_Sqh}
 \mathfrak{S}_{q,h} &=& \prod_{\substack{P\in\mathbb{F}_q[T]\\\mathrm{monic\; irred.}}}\frac{\delta_{q,h}(P)}{\delta_{q,0}(P)^k}
\end{eqnarray}
with
\[
\delta_{q,h}(P) = \lim_{\nu\rightarrow\infty}\frac{\#\{f\in\mathbb{F}_q[T]/(P^\nu) : 
\forall i\exists A_i,B_i\: f+h_i\equiv A_i^2+TB_i^2\mod P^{\nu}\}}{|P|^{\nu}}.
\]
\end{conjecture}
The constant $\mathfrak{S}_{q,h}$ comes from the same heuristics that led to  the constant $\mathfrak{S}_h$ of \eqref{eq:autocorrelation}, see Section~\ref{sec:heuristics} for details, so Conjecture~\ref{conj} is in perfect analogy with \eqref{eq:autocorrelation}. 
We note that the product  $\mathfrak{S}_{q,h}$ always converges and $\mathfrak{S}_{q,h}>0$ if and only if there are no local obstructions 
(Corollary~\ref{cor:localobstruction}). 

\subsection{Results and method}
As mentioned before, the main goal of this work is to give evidence for \eqref{eq:autocorrelation}
by studying its function field analogue, Conjecture \ref{conj}.
We show that the local factors $\delta_{q,h}(P)$ can be computed in theory and we carry out this computation in special cases like $k=2$. 
In Section~\ref{sec:numerics}, we provide numerical evidence that supports Conjecture~\ref{conj}, 
which is more extensive than what can be done in number fields, as there is a fast algorithm to compute the analogue of $b$.
We then prove Conjecture~\ref{conj} in the large finite field limit,
where the main term can be explicitly computed:

\begin{theorem}\label{thm:b}
Fix $n\geq 3$ and $k\geq 1$.
Then for $q$ an odd prime power
and $h_1,\dots,h_k\in\mathbb{F}_q[T]$ of degree less than $n$ and pairwise distinct,
\begin{eqnarray*}
  \left< \prod_{i=1}^k b_q(f+h_i) \right>_{f\in M_{n,q}} &=&
 \mathfrak{S}_{q,h} \cdot\left< b_q(f) \right>_{f\in M_{n,q}}^k + O_{n,k}(q^{-1/2}) \\
 &=&  \mathfrak{S}_h\cdot \frac{1}{4^{nk}}\left(2n\atop n\right)^k +O_{n,k}(q^{-1/2}) ,
\end{eqnarray*}
where the implied constant depends only on $n$ and $k$, and
\begin{eqnarray}
\label{eqn:def_Sh}
 \mathfrak{S}_h &=& 2^{k-\#\{h_1(0),\dots,h_k(0)\}}.
\end{eqnarray}
\end{theorem}

We would like to highlight an interesting phenomenon: 
In the previously mentioned results on arithmetic functions that depend only on the cycle structure, 
 it was shown that in the large finite field limit they become independent; 
hence the correlation dependence on $h$ disappears and can be seen only in the error term, cf.\ \cite{KRG}. 
However, here, the dependence on the $h_i$ is non-trivial and agrees with the heuristics. 
The simple form of $\mathfrak{S}_h$ (as opposed to $\mathfrak{S}_{q,h}$) can be read as saying that in the large finite field limit only the correlation modulo the prime $T$ remains,
as $f\equiv A^2+TB^2\mbox{ mod }T$ for some $A,B\in\mathbb{F}_q[T]$ if and only if $f(0)$ is a square in $\mathbb{F}_q$.

In fact we prove Theorem \ref{thm:b} for correlations in {\em short intervals} instead of $M_{n,q}$ (see Theorem \ref{thm:b_short})
and we also phrase Conjecture \ref{conj} in this generality (see Conjecture \ref{conj_short}).
We derive Theorem \ref{thm:b} from a general result on correlations of arithmetic functions
that depend on what we call {\em signed factorization type}. This result reduces the large finite field limit of these correlations to combinatorial problems
in certain finite groups; namely, fiber products of hyperoctahedral groups.
In Section \ref{sec:signed}, we explain this general result (Theorem~\ref{thm:general})
and its proof, the main part of which consists of 
showing that the Galois group of a certain polynomial with a few variable coefficients
is a fiber product of hyperoctahedral groups.
Section \ref{sec:applications} then contains
the combinatorics for $b_q$, leading to Theorem \ref{thm:b},
as well as of a few other arithmetic functions.
In particular, we compute (function field analogues of) the average of $b$ on shifted primes (Theorem \ref{thm:r_Lambda}),
and autocorrelations of higher divisor functions twisted by a quadratic character (Theorem \ref{thm:dkchi});
see the corresponding sections for more on the history and motivation of these questions.

\pagebreak

\section{Conjectural correlation of sums of two squares}
\label{sec:heuristics}

\subsection{Preliminaries}
\label{sec:prelim}

With the convention $\deg 0=-\infty$, the norm on $\FF_q[T]$ is given by 
$|f| = q^{\deg f}$.
A short interval around a polynomial $f_0$ of degree $\deg f_0=n$ is defined analogously to the short intervals of integers $|n-x|< x^{\epsilon}$,
cf.~\cite{KeatingRudnick}: 
\[
I_{q}(f_0,\epsilon)
:= \{ f\in\mathbb{F}_q[T]: |f-f_0|<|f_0|^{\epsilon}\} 
= \left\{f_0 + \sum_{i<\epsilon n} a_iT^i : a_0,\dots,a_{\lfloor\epsilon n\rfloor}\in\mathbb{F}_q\right\}.
\]
So, roughly speaking, $f\in I_{q}(f_0,\epsilon)$ if and only if $1-\epsilon$ fraction of its higher coefficients coincide with those of $f_0$. 

Every prime polynomial $P\in M_q$ defines a unique $P$-adic valuation $v_P\colon \mathbb{F}_q(T)\rightarrow\mathbb{Z}\cup\{\infty\}$
with $v_P(P)=1$.
The condition that a prime number $p$ is inert in $\mathbb{Z}[i]$ corresponds  to the condition that a prime polynomial $P\in \FF_q[T]$ is irreducible in $\FF_q[\sqrt{-T}]$; equivalently, that $P(-T^2)$ is irreducible in $\mathbb{F}_q[T]$. 
Hence, one of the reasons why $b_q$ is a suitable function field analogue of $b$ is that
it has a multiplicative description very similar to Fermat's multiplicative description of $b$:

\begin{proposition}\label{prop:Fermat}
Let $f\in M_{q}$. Then $b_q(f)=1$ if and only if 
$v_P(f)\equiv 0\mod 2$
for every prime polynomial $P\in\mathbb{F}_q[T]$ with $P(-T^2)\in\mathbb{F}_q[T]$ irreducible.
\end{proposition}

\begin{proof}
See \cite[Thm.~2.5]{BSW}.
\end{proof}

If we denote by 
\begin{eqnarray}
\label{eqn:chi_q}
 \chi_q\colon \mathbb{F}_q[T]\rightarrow\{\pm1,0\}, \qquad \chi_q(f) = \begin{cases}
 0, & \mbox{if }f(0)=0,\\
 1, & \mbox{if }f(0)\in\mathbb{F}_q^{\times2}, \\
 -1, & \mbox{otherwise.}
 \end{cases}
\end{eqnarray}
the quadratic Dirichlet character modulo $T$, 
then the condition that $P(-T^2)$ is irreducible is equivalent to $\chi_q(P)=-1$,
see \cite[\S2]{BSW}.

\subsection{Heuristics}

The notation of this and the following subsection is taken from a preliminary version of \cite{FKR},
and all proofs follow the corresponding proofs there.
For $P\in\mathbb{F}_q[T]$ monic irreducible and $\nu\in\mathbb{N}$ we define $\mathcal{A}_q(P^\nu)\subseteq\mathbb{F}_q[T]/(P^\nu)$ by
\begin{eqnarray*}
 \mathcal{A}_{q}(P^\nu) &=& \{ f + (P^\nu) : f\in M_q, b_q(f)=1 \}  \\
 &=& \{A^2+TB^2 : A,B\in\mathbb{F}_q[T]/(P^\nu)\}. 
\end{eqnarray*}
Moreover, for $h=(h_1,\dots,h_k)\in\mathbb{F}_q[T]^k$ we let
\begin{eqnarray*} 
\mathcal{A}_{q,h}(P^\nu) &=& \{ f\in\mathbb{F}_q[T]/(P^\nu) : f+h_1,\dots,f+h_k\in\mathcal{A}_q(P^\nu)\}.
\end{eqnarray*}
Note that if $f\in\mathcal{A}_{q,h}(P^\nu)$ then trivially $f\in\mathcal{A}_{q,h}(P^\xi)$ for all $\xi\leq\nu$,
hence the sequence $|P|^{-\nu}\#\mathcal{A}_{q,h}(P^\nu)$ is monotone decreasing and therefore the limit
\begin{eqnarray*}
 \delta_{q,h}(P) = \lim_{\nu\rightarrow\infty}|P|^{-\nu}\#\mathcal{A}_{q,h}(P^\nu)
\end{eqnarray*}
exists. We will also show below (Corollary~\ref{cor:localobstruction}) that $\delta_{q,h}(P)=0$ if and only if $\mathcal{A}_{q,h}(P^{\nu})=\emptyset$ for some $\nu>0$ in which case we say that there exists \emph{local obstruction} at $P$.

Our heuristic assumption is that the $b_q(f+h_i)$, $i=1,\ldots, k$, behave like i.i.d.\ random variables as $f$ is randomly picked from a short interval $|f-f_0|< |f_0|^{\epsilon}$, up to a correction factor coming  from the fact that they are not independent modulo polynomials $g$. 
By the Chinese Remainder Theorem, one may reduce to $g=P^{\nu}$ a prime power, for which the actual mean of $\prod_ib_q(f+h_i)$ is given 
by $\delta_{q,h}(P)$ while the random model would predict a mean of $\delta_0(P)^k$.  
This leads to
\begin{conjecture}\label{conj_short}
For every $k\geq 1$, $d\geq 1$ and $1\geq \epsilon>0$ there exists $N\geq 1$ such that
for $q$ an odd prime power, 
$f_0\in\mathbb{F}_q[T]$ monic of degree $n\geq N$ and $h_1,\dots,h_k\in\mathbb{F}_q[T]$ of degree less than $d$ and pairwise distinct,
\begin{eqnarray}\label{eq:CONJ}
\left<\prod_{i=1}^k b_q(f+h_i)\right>_{|f-f_0|<|f_0|^\epsilon} &\sim& \mathfrak{S}_{q,h}\cdot \left<b_q(f)\right>_{f\in M_{n,q}}^k\\
&\sim& \mathfrak{S}_{q,h}\cdot K_q^k\cdot\frac{1}{4^{nk}}\left(2n\atop n\right)^k
\end{eqnarray}
uniformly as $q^n\rightarrow\infty$, where
$\mathfrak{S}_{q,h}$ and $K_q$ are defined as in \eqref{eqn:def_Sqh} resp.~\eqref{eqn:def_Kq}.
\end{conjecture}
We note that if there exists a local obstruction, then clearly both sides in \eqref{eq:CONJ} 
equal $0$, so in this case the conjecture is uninteresting and correct.

\subsection{The singular series}
We now show that the $\delta_{q,h}(P)$ can be computed in theory and give estimates in general and concrete formulas in certain special cases. In particular, we show that the infinite product in \eqref{eqn:def_Sqh} that defines $\mathfrak{S}_{q,h}$ indeed converges.

\begin{lemma}\label{lem:characterize}
Let $P\in\mathbb{F}_q[T]$ be monic irreducible and $f\in\mathbb{F}_q[T]$.
Assume that $f\not\equiv 0\mod P^\nu$ and
write $f\equiv P^\alpha g\mod P^\nu$ with maximal $0\leq\alpha<\nu$ and any suitable $g\in\mathbb{F}_q[T]$.
\begin{enumerate}
\item If $\chi_q(P)=-1$, then $f\in\mathcal{A}_q(P^\nu)$ if and only if $\alpha$ is even.
\item If $\chi_q(P)=0$, then $f\in\mathcal{A}_{q}(P^\nu)$ if and only if $\chi_q(g)=1$.
\item If $\chi_q(P)=1$, then $f\in\mathcal{A}_q(P^\nu)$.\label{lem:characterize3}
\end{enumerate}
\end{lemma}

\begin{proof}
Note that $\mathcal{A}_q(P^\nu)$ is closed under multiplication
and contains $P$ if $\chi_q(P)\in\{0,1\}$, and $P^2$ if $\chi_q(P)=-1$.
If $\chi_q(P)\in\{\pm1\}$, then $g\in\mathcal{A}_q(P^\nu)$:
This follows for example from
the prime polynomial theorem in arithmetic progressions \cite{Rosen}
which gives a prime polynomial $Q$ with $Q\equiv g\mod P^\nu$ and $Q\equiv 1\mod T$,
hence $b_q(Q)=1$ (Proposition \ref{prop:Fermat}).
This proves (3) and the `if' part of (1).
For the `only if' part of (1),
note that 
if $b_q(f)=1$ and $f\equiv P^\alpha g\mod P^\nu$, then $v_P(f)=\alpha$ is even (Proposition \ref{prop:Fermat}).
For the `if' part of (2) we can again apply the prime polynomial theorem to get
a prime $Q$ with $Q\equiv g\mod P^\nu$, which then satisfies $\chi_q(Q)=\chi_q(g)=1$, hence $b_q(Q)=1$.
The `only if' part of (2) is obvious, since the lowest nonzero coefficient of $A^2+TB^2$ is a square.
\end{proof}

\begin{lemma}
\label{lem:size_of_Aq}
Let $P\in\mathbb{F}_q[T]$ be monic irreducible.
\begin{enumerate}
\item If $\chi_q(P)=-1$, then 
$$
 \#\mathcal{A}_q(P^\nu)=|P|^\nu\left(1-\frac{1}{|P|+1}\right)+\begin{cases}\frac{|P|}{|P|+1},&\nu\mbox{ odd}\\
\frac{1}{|P|+1},&\nu\mbox{ even}\end{cases}.
$$
\item If $\chi_q(P)=0$, then $\#\mathcal{A}_q(P^\nu) = \frac{|P|^\nu+1}{2}$.
\item If $\chi_q(P)=1$, then $\#\mathcal{A}_q(P^\nu)=|P|^\nu$.
\end{enumerate}
\end{lemma}

\begin{proof}
This follows by direct counting using Lemma \ref{lem:characterize}.
For example, if $\chi_q(P)=-1$ and $\nu$ is odd, then
the nonzero elements of $\mathcal{A}_q(P^\nu)$ are represented by polynomials
$\sum_{i=0}^{\nu-1} a_iP^i$ with ${\rm deg}(a_i)<{\rm deg}(P)$
and $\min\{i:a_i\neq 0\}$ even. Thus,
\begin{eqnarray*}
 \#\mathcal{A}_q(P^\nu)-1 &=& \sum_{\alpha=0\,{\mathrm even}}^{\nu-1}(|P|-1)|P|^{\nu-\alpha-1} = 
 |P|^{\nu-1}(|P|-1)\frac{1-|P|^{-\nu-1}}{1-|P|^{-2}} ,
\end{eqnarray*}
from which the claim follows.
\end{proof}

Lemma~\ref{lem:size_of_Aq} immediately gives $\delta_{q,h}(P)$ in the special case $h=0$ (which we identify with the $1$-tuple  $(0)$) or $\chi_q(P)=1$:

\begin{corollary}\label{cor:deltah_chi_1}
Let $P\in\mathbb{F}_q[T]$ monic irreducible. 
If $\chi_q(P)=1$, then $\delta_{q,h}(P)=1$.
\end{corollary}

\begin{corollary}\label{lem:delta0}
Let $P\in\mathbb{F}_q[T]$ monic irreducible. Then
\[
\delta_{q,0}(P) = 
	\begin{cases}
		1-\frac{1}{|P|+1},		&\chi_q(P)=-1\\
		\frac{1}{2},	&\chi_q(P)=0\\
         1,&\chi_q(P)=1.
	\end{cases}
\]
\end{corollary}

The computation of $\delta_{q,h}(P)$ in the rest of the cases is more technical. For $h=(h_1,\ldots, h_k)$ let
\begin{eqnarray}
\label{eq:Delta} \Delta_h &=& \prod_{i\neq j}(h_i-h_j),\\
\label{eq:nu} \nu_h(P) &=& \max_{i\neq j}v_P(h_i-h_j).
\end{eqnarray}

\begin{lemma}\label{lem:Aestimate}
Let $P\in\mathbb{F}_q[T]$ be monic irreducible.
Let $h=(h_1,\dots,h_k)\in\mathbb{F}_q[T]^k$ be a $k$-tuple of pairwise distinct polynomials.
If $\chi_q(P)=-1$, then
$$
 |\mathcal{A}_{q,h}(P^\nu)|\geq|P|^\nu(1-\frac{k}{|P|+1})+k\cdot\begin{cases}\frac{|P|}{|P|+1},&\nu\mbox{ odd}\\\frac{1}{|P|+1},&\nu\mbox{ even}\end{cases}
$$
with equality if $P\nmid\Delta_{h}$.
\end{lemma}

\begin{proof}
We use Lemma \ref{lem:size_of_Aq} and that 
$\mathcal{A}_{q,h}(P^\nu)=\bigcap_{i=1}^k\mathcal{A}_{q,h_i}(P^\nu)$
and $\#\mathcal{A}_{q,h_i}(P^\nu)=\#\mathcal{A}_{q}(P^\nu)$,
to get that 
$$
 \#\mathcal{A}_{q,h}(P^\nu)=|P|^\nu-\#(\mathbb{F}_q[T]/(P^\nu)\setminus\mathcal{A}_{q,h}(P^\nu))\geq|P|^\nu-k(|P|^\nu-\#\mathcal{A}_{q}(P^\nu)).
$$
If $P\nmid\Delta_h$, then at most one of $f+h_1,\dots,f+h_k$ is not in $\mathcal{A}_q(P^\nu)$ 
(by Lemma \ref{lem:characterize}(1)), 
hence the complements of the $\mathcal{A}_{q,h_i}(P^\nu)$ are disjoint.
\end{proof}

\begin{proposition}\label{lem:delta_chi_-1}
If $\chi_q(P)=-1$, then 
$$
 \delta_{q,h}(P)=1-\frac{\eta}{|P|+1}
$$
where $1\leq\eta\leq k$.
Moreover, $\eta=1$ if $h=0$, and $\eta=k$ if $P\nmid\Delta_h$.
\end{proposition}

\begin{proof}
Trivially, $\mathcal{A}_{q,h}(P^\nu)\subseteq\mathcal{A}_{q}(P^\nu)$, hence $\delta_{q,h}(P)\leq\delta_{q,0}(P)$ is an upper bound,
and $\delta_{q,0}(P)=1-\frac{1}{|P|+1}$ by Corollary \ref{lem:delta0}.
The lower bound follows from Lemma \ref{lem:Aestimate}.
\end{proof}

We let
\[
 \tilde{\mathfrak{S}}_{q,k}:=\prod_{\chi_q(P)=-1}\frac{1-\frac{k}{|P|+1}}{\delta_{q,0}(P)^k}
\]
and observe that 
\begin{equation}
\label{eqn:S_tildeS}
\mathfrak{S}_{q,h} = \frac{\delta_{q,h}(T)}{\delta_{q,0}(T)^k}\cdot\prod_{P|\Delta_h,\chi_q(P)=-1}\frac{\delta_{q,h}(P)}{1-\frac{k}{|P|+1}}\cdot\tilde{\mathfrak{S}}_{q,k}
\end{equation}
by Proposition \ref{lem:delta_chi_-1} and Corollary \ref{cor:deltah_chi_1}.
So $\mathfrak{S}_{q,h}$ is convergent if and only if $\tilde{\mathfrak{S}}_{q,k}$ is convergent. 

\begin{lemma}\label{lem:productconverges}
For every $k$, the product defining $\tilde{\mathfrak{S}}_{q,k}$  converges to a positive constant,
and $\tilde{\mathfrak{S}}_{q,k} = 1 + O_k(q^{-1})$.
\end{lemma}

\begin{proof}
For each prime polynomial $P$ put 
\[
a_P = k \log (1-\frac{1}{|P|+1})-\log(1-\frac{k}{|P|+1}).
\] 
By Bernoulli's inequality we have that $a_P\geq 0$ for all $P$, hence the series 
\[
S_q := \sum_{\chi_q(P)=-1} a_P
\]
in convergent if and only if it is bounded. 
Using 
the Taylor expansion $-\log (1-x)=x+O(x^2)$, we have that 
\begin{equation*}\label{eq:a_P}
a_P \ll |P|^{-2} = q^{-2\deg P}.
\end{equation*}
Since the number of $P$ of degree $d$ with $\chi_q(P)=-1$ is trivially no more than $q^d$, this gives 
\[
S_q  \ll \sum_{d=1}^\infty q^d q^{-2d} \ll \frac{1}{q} .
\]
Thus the series is convergent and $S_q=O(1/q)$.
Now, as $\tilde{\mathfrak{S}}_{q,k}=\exp(-S_q)$,
the product $\tilde{\mathfrak{S}}_{q,k}$ converges and $\tilde{\mathfrak{S}}_{q,k}=1+O(1/q)$.
\end{proof}

From Lemma~\ref{lem:productconverges} and \eqref{eqn:S_tildeS} the following assertion immediately follows:

\begin{corollary}\label{cor:convergence}
Let $h=(h_1,\ldots, h_k) \in \FF_q[T]^k$ be a $k$-tuple of pairwise distinct polynomials and let $\Delta_h$ as defined in \eqref{eq:Delta}. Then $\mathfrak{S}_{q,h}$ converges. Moreover $\mathfrak{S}_{q,h}=0$ if and only if there exists $P \mid T \Delta_h(T)$ with $\chi_q(P)\neq 1$ such that  $\delta_{q,h}(P)=0$.
\end{corollary}

Next we give convenient formulas for $\mathfrak{S}_{q,h}$.

\begin{lemma}
\label{lem:lift}
Let $P\in\mathbb{F}_q[T]$ monic irreducible and $f\in\mathcal{A}_{q}(P^\nu)$.
\begin{enumerate}
\item Assume $\chi_q(P)=-1$.
\begin{enumerate}
\item If $f\not\equiv 0\mod P^\nu$ or $\nu$ is even, then $f\in\mathcal{A}_{q}(P^{\nu+1})$.
\item If $f\equiv 0\mod P^\nu$ and $\nu$ is odd, then $f\in\mathcal{A}_{q}(P^{\nu+1})$ if and only if
$f\equiv 0\mod P^{\nu+1}$.
\end{enumerate}
\item Assume $\chi_q(P)=0$.
\begin{enumerate}
\item If $f\not\equiv 0\mod P^\nu$, then $f\in\mathcal{A}_{q}(P^{\nu+1})$.
\item If $f\equiv 0\mod P^\nu$, then $f\equiv\alpha P^\nu\mod P^{\nu+1}$ with $\alpha\in\mathbb{F}_q$,
and $f\in\mathcal{A}_{q}(P^{\nu+1})$ if and only if $\alpha=0$ or $\alpha\in\mathbb{F}_q^{\times2}$.
\end{enumerate}
\end{enumerate}
\end{lemma}

\begin{proof}
This is immediate from Lemma \ref{lem:characterize}.
\end{proof}

We define
\begin{equation}\label{eq:def_Omegas}
\begin{split}
 \Omega_{q,h}(P^\nu) &= \{ f\in\mathcal{A}_{q,h}(P^\nu) : \exists i\; f+h_i \equiv 0\mod P^\nu\},\\
 \Omega_{q,h}^*(P^\nu) &= \{ f\in\mathcal{A}_{q,h}(P^\nu) : \forall i\; f+h_i \not\equiv 0\mod P^\nu\}.
\end{split}
\end{equation}
Note that if $\nu>\nu_h(P)$, then $h_1+(P^\nu),\dots,h_k+(P^\nu)$ are pairwise distinct,
so $f+h_i\equiv 0\mod P^\nu$ for at most one $i$, and
\begin{equation}\label{eq:formula_Omega}
 \#\Omega_{q,h}(P^\nu) = \#\{i : h_j-h_i\in\mathcal{A}_q(P^\nu) \mbox{ for all }j\}.
\end{equation}

\begin{lemma}\label{lem:liftwithh}
Let $P\in\mathbb{F}_q[T]$ be monic irreducible and $f\in\mathcal{A}_{q}(P^\nu)$, let $h=(h_1,\ldots, h_k)\in \FF_q[T]^k$ be a $k$-tuple of pairwise distinct polynomials, and $\nu_h$ as defined in \eqref{eq:nu}.
\begin{enumerate}
\item \label{eq:liftwithh1}
Assume $\chi_q(P)=-1$. 
\begin{enumerate} 
\item If $f\in\Omega_{q,h}^*(P^\nu)$, then $f\in\Omega_{q,h}^*(P^\xi)$ for all $\xi\geq\nu$.
\item 
If $f\in\Omega_{q,h}(P^\nu)$ and $\nu>\nu_h(P)$, then $f+gP^\nu\in\Omega_{q,h}(P^{\nu+1})$ for precisely one $g\in\mathbb{F}_q[T]$ 
with ${\rm deg}(g)<{\rm deg}(P)$. For all $g'\neq g$ with ${\rm deg}(g')<{\rm deg}(P)$,
$f+g'P^\nu\in\Omega_{q,h}^*(P^{\nu+1})$ if $\nu$ is even,
and $f+g'P^\nu\notin\mathcal{A}_{q,h}(P^{\nu+1})$ if $\nu$ is odd.
\end{enumerate}
\item  \label{eq:liftwithh2}
Assume $\chi_q(P)=0$.
\begin{enumerate}
\item If $f\in\Omega_{q,h}^*(P^\nu)$, then $f\in\Omega_{q,h}^*(P^\xi)$ for all $\xi\geq\nu$.
\item If $f\in\Omega_{q,h}(P^\nu)$ and $\nu>\nu_h(P)$, then $f+\alpha P^\nu\in\Omega_{q,h}(P^{\nu+1})$ for one $\alpha\in\mathbb{F}_q$,
$f+\alpha P^\nu\in\Omega_{q,h}^*(P^{\nu+1})$ for $\frac{q-1}{2}$ many $\alpha\in\mathbb{F}_q$,
and $f+\alpha P^\nu\notin\mathcal{A}_{q,h}(P^{\nu+1})$ for $\frac{q-1}{2}$ many $\alpha\in\mathbb{F}_q$.
\end{enumerate}
\end{enumerate}
\end{lemma}

\begin{proof}
This follows by applying Lemma \ref{lem:lift} to the $f+h_i$.
\end{proof}

\begin{proposition}\label{prop:deltaformula}
Let $P\in\mathbb{F}_q[T]$ be monic irreducible, let $h=(h_1,\dots,h_k)\in\mathbb{F}_q[T]^k$ be a $k$-tuple of   pairwise distinct polynomials, and $\nu_h$ as in \eqref{eq:nu}.
Fix $\nu>\nu_h(P)$.
\begin{enumerate}
\item 
If $\chi_q(P)=-1$, then 
$$
 \delta_{q,h}(P)=\begin{cases}
  |P|^{-\nu}(\#\Omega_{q,h}^*(P^\nu)+\frac{1}{|P|+1}\#\Omega_{q,h}(P^\nu)), &\nu\mbox{ odd}\\
  |P|^{-\nu}(\#\Omega_{q,h}^*(P^\nu)+\frac{|P|}{|P|+1}\#\Omega_{q,h}(P^\nu)), &\nu\mbox{ even}
  \end{cases}
$$
\item \label{eq:deltaformulaT}
If $\chi_q(P)=0$, then
$$
 \delta_{q,h}(P)= |P|^{-\nu}(\#\Omega_{q,h}^*(P^\nu)+\frac{1}{2}\#\Omega_{q,h}(P^\nu)). 
$$
\end{enumerate}
\end{proposition}

\begin{proof}
(1): By Lemma \ref{lem:liftwithh}\eqref{eq:liftwithh1}, for each $\xi\geq\nu$, we have $\#\Omega_{q,h}(P^\xi)=\#\Omega_{q,h}(P^\nu)$ and
$$
 \#\Omega_{q,h}^*(P^{\xi+1}) = |P|\cdot\#\Omega_{q,h}^*(P^\xi) + \begin{cases}(|P|-1)\#\Omega_{q,h}(P^\xi)&\xi\mbox{ even}\\0&\xi\mbox{ odd}\end{cases}
$$
so, as $\xi\rightarrow\infty$,  $|P|^{-\xi}\#\Omega_{q,h}^*(P^\xi)$ tends to
$$
  |P|^{-\nu}\#\Omega_{q,h}^*(P^\nu)+|P|^{-\nu}\#\Omega_{q,h}(P^\nu)\cdot(|P|-1)\cdot
 \begin{cases}
  \sum_{\mu=0}^\infty|P|^{-(2\mu+1)}&\nu\mbox{ even}\\\sum_{\mu=1}^\infty|P|^{-2\mu}&\nu\mbox{ odd}\end{cases},
$$
from which the claim follows, as $|P|^{-\xi}(\#\mathcal{A}_{q,h}(P^\xi)-\#\Omega^*_{q,h}(P^\xi))=|P|^{-\xi}\#\Omega_{q,h}(P^\nu)\rightarrow 0$,
since $\#\Omega_{q,h}(P^\nu)\leq k$.

(2): By Lemma \ref{lem:liftwithh}\eqref{eq:liftwithh2}, for each $\xi\geq\nu$, we have $\#\Omega_{q,h}(P^\xi)=\#\Omega_{q,h}(P^\nu)$ and
$$
 \#\Omega_{q,h}^*(P^{\xi+1}) = |P|\cdot\#\Omega_{q,h}^*(P^\xi) + \frac{|P|-1}{2}\cdot\#\Omega_{q,h}(P^\xi),
$$
so, as $\xi\rightarrow\infty$,  $|P|^{-\xi}\#\Omega_{q,h}^*(P^\xi)$ tends to
$$
 |P|^{-\nu}\#\Omega_{q,h}^*(P^\nu)+|P|^{-\nu}\#\Omega_{q,h}(P^\nu)\cdot\frac{|P|-1}{2}\cdot\sum_{\mu=1}^\infty|P|^{-\mu},
$$
from which again the claim follows.
\end{proof}

Note the striking similarity between Proposition~\ref{prop:deltaformula}(\ref{eq:deltaformulaT}) and the formula for $\delta_{\mathbf{k}}(2)$ in \cite{FKR}.
As an immediate consequence of Corollary~\ref{lem:delta0} and  Proposition~\ref{prop:deltaformula} we have 
\begin{corollary}\label{cor:localobstruction}
There exists local obstruction at $P$ if and only if $\delta_{q,h}(P)=0$.  
In particular, $\mathfrak{S}_{q,h}=0$ if and only if there exists local obstruction at some prime $P$.
\end{corollary}
We now give explicit formulas for the case $k=2$; i.e.~$h=(h_1,h_2)$.
In this case, we may assume that $h$ is of the form $h=(0,h_1)$.

\begin{proposition}\label{prop:Sfor0h}
For $h=(0,h_1)$ and $P\in\mathbb{F}_q[T]$ monic irreducible with $v_P(h_1)=\rho$ we have
\begin{eqnarray*}
 \delta_{q,h}(P) = \begin{cases} 
  1-\frac{|P|^\rho+1}{|P|^\rho(|P|+1)},&\chi_q(P)=-1\\
  \frac{1}{2}-\frac{1}{4|P|^\rho}-\frac{1}{4|P|^{\rho+1}},&\chi_q(P)=0\\
  1,&\chi_q(P)=1.\\
\end{cases}
\end{eqnarray*}
In particular, $\delta_{q,h}(P) >0$ for all $P$,
so there exists no local obstruction in the case $k=2$.
\end{proposition}

\begin{proof}
Apply Proposition \ref{prop:deltaformula} with $\nu=\rho+1$.
If $\chi_q(P)=-1$ and $\nu$ is odd, then $\#\Omega_{q,h}(P^\nu)=2$ and $\#\Omega^*_{q,h}(P^\nu)=\frac{|P|^\nu(1-|P|^{-\rho})}{1+|P|^{-1}}+|P|-2$.
If $\chi_q(P)=-1$ and $\nu$ is even, then $\#\Omega_{q,h}(P^\nu)=0$ and $\#\Omega^*_{q,h}(P^\nu)=\frac{|P|^\nu-1}{1+|P|^{-1}}$.
If $\chi_q(P)=0$ and $h=T^\rho g$, then $\#\Omega_{q,h}(P^\nu)=1+\frac{1}{2}(\chi_q(g)+\chi_q(-g))$
and $\#\Omega_{q,h}^*(P^\nu)=\frac{q^\nu}{2}-\frac{q}{2}+\frac{1}{4}(q-3-\chi_q(g)-\chi_q(-g))$.
\end{proof}

\begin{example}
For $h=(0,1)$, we get
$$
 \delta_{q,h}(P) = \begin{cases} 
 1-\frac{2}{|P|+1},&\chi_q(P)=-1\\
 \frac{1}{4}-\frac{1}{4|P|},&\chi_q(P)=0\\
  1,&\chi_q(P)=1\end{cases}
$$
For $h=(0,T)$, we get
$$
 \delta_{q,h}(P) = \begin{cases} 
 1-\frac{2}{|P|+1},&\chi_q(P)=-1\\
 \frac{1}{2}-\frac{1}{4|P|}-\frac{1}{4|P|^2},&\chi_q(P)=0\\
 1,&\chi_q(P)=1 \end{cases}
$$
In particular, we conclude that
$$
 \frac{\mathfrak{S}_{q,(0,T)}}{\mathfrak{S}_{q,(0,1)}} = \frac{\delta_{q,(0,T)}(T)}{\delta_{q,(0,1)}(T)} =  \frac{2q^2-q-1}{q(q-1)}=2+\frac{1}{q},
$$
so Conjecture \ref{conj_short} predicts that
\begin{equation}\label{eq:1T}
 \frac{\left<b_q(f)b_q(f+T)\right>_{f\in M_{n,q}}}{\left<b_q(f)b_q(f+1)\right>_{f\in M_{n,q}}} \sim 2+\frac{1}{q},\quad q^n\rightarrow\infty.
\end{equation}
\end{example}

\begin{remark}
We note that, as opposed to the situation in integers, 
for small $q$ there can be a local obstruction in the case $k=3$,
e.g.~if $q=3$ and $h=(0,1,2)$, then $\delta_{q,h}(T)=0$.
\end{remark}

\begin{proposition}\label{prop:limSqh}
Fix $k\geq 1$ and $d\geq 1$.
For $h_1,\dots,h_k\in\mathbb{F}_q[T]$ pairwise distinct and of degree less than $d$,
\begin{eqnarray*}
  \mathfrak{S}_{q,h} = \mathfrak{S}_h+O_{k,d}(q^{-1/2}),
\end{eqnarray*}
where the implied constant depends only on $k$ and $d$,
and $\mathfrak{S}_{q,h}$ and $\mathfrak{S}_h$ are defined as in (\ref{eqn:def_Sqh}) and (\ref{eqn:def_Sh}).
\end{proposition}

\begin{proof}
Applying Lemma \ref{lem:productconverges},
Proposition \ref{lem:delta_chi_-1}, Corollary \ref{lem:delta0}
to (\ref{eqn:S_tildeS}) and noting that $|P|\geq q$ gives
$$
 \mathfrak{S}_{q,h} = \frac{\delta_{q,h}(T)}{\delta_{q,0}(T)^k}\cdot\prod_{P|\Delta_h,\chi_q(P)=-1}\frac{\delta_{q,h}(P)}{1-\frac{k}{|P|+1}}\cdot\tilde{\mathfrak{S}}_{q,k}
 =2^k\delta_{q,h}(T)+O(q^{-1}).
$$
Let $S=\{h_1(0),\dots,h_k(0)\}$ and $l=\#S$.
Since on $\mathbb{F}_q^\times$, 
$\frac{\chi_q+1}{2}$ is the indicator function of $\mathcal{A}_q(T)$, 
by the definition of $\Omega^*_{q,h}$ in \eqref{eq:def_Omegas} we can write 
$$
 \#\Omega^*_{q,h}(T)=\sum_{\substack{\alpha\in\mathbb{F}_q\\-\alpha\notin S}}\prod_{\beta\in S}\frac{\chi_q(\alpha+\beta)+1}{2}
  = 2^{-l} \sum_{S'\subseteq S}\sum_{\substack{\alpha\in\mathbb{F}_q\\-\alpha\notin S}}\chi_q(\prod_{\beta\in S'}(\alpha+\beta)).
$$
The Hasse-Weil theorem gives that $\sum_{\substack{\alpha\in\mathbb{F}_q\\-\alpha\notin S}}\chi_q(\prod_{\beta\in S'}(\alpha+\beta))=O(q^{-1/2})$
when $S'\neq \emptyset$. When $S'=\emptyset$, one trivially has $\sum_{\substack{\alpha\in\mathbb{F}_q\\-\alpha\notin S}}\chi_q(\prod_{\beta\in S'}(\alpha+\beta)) = q-l$. Thus,
$\#\Omega^*_{q,h}(T)=\frac{q-l}{2^l}+O(q^{1/2})$ and
$$
 q^{-\nu}\#\Omega^*_{q,h}(T^\nu)=q^{-\nu}\cdot(q^{\nu-1}\#\Omega^*_{q,h}(T)+O(q^{\nu-1}))= \frac{1}{2^l}+O(q^{-1/2}).
$$ 
By Proposition \ref{prop:deltaformula}, $\delta_{q,h}(T)=2^{-l}+O(q^{-1/2})$, as claimed.
\end{proof}

We conclude this section by pointing out that $\delta_{q,h}(P)$ could be defined (and possibly computed) also in different ways,
for example as the asymptotic density of
\begin{eqnarray*} 
\mathcal{A}_{q,h}(P^\infty) &=& \{ f\in\mathbb{F}_q[T] : f+(P^\nu)\in\mathcal{A}_{q,h}(P^\nu)\mbox{ for all }\nu\}
\end{eqnarray*}
like in \cite{FKR}, or as the Haar measure of a suitable set in the completion of $\FF_q[T]$ at $P$.

%
%
%

\subsection{Numerics}
\label{sec:numerics}
In this section we compare Conjecture \ref{conj_short} with numerical computations. 
All our computation were done with the SageMath mathematics software system \cite{Sage}.
The algorithm to compute the function $b_q(f)$ is based on factoring $f$ into irreducibles and then applying Proposition~\ref{prop:Fermat}.
Here we exploit the fact that there is a fast factorization algorithm for $\mathbb{F}_q[T]$ implemented, 
as opposed to the situation in integers.

In the computations of the singular series, we need to compute $\delta_{q,h}(P)$, for which we apply 
Proposition~\ref{lem:delta_chi_-1} and Proposition~\ref{prop:deltaformula}, among other results.

We enumerate the polynomials in $\mathbb{F}_q[T]$ in the following order: $0,1,\dots,T,T+1,\dots$,
and we will use this enumeration also in what follows.

\subsubsection{Varying $h$}
Conjecture \ref{conj} predicts that for each $h\in\mathbb{F}_q[T]$,
\begin{equation}\label{eq:varyinghconj}
 \frac{\left<b_q(f)b_q(f+h)\right>_{f\in M_{n,q}}}{\left<b_q(f)b_q(f+1)\right>_{f\in M_{n,q}}} \sim \frac{\mathfrak{S}_{q,(0,h)}}{\mathfrak{S}_{q,(0,1)}},\quad n\rightarrow\infty.
\end{equation}
In Figures~\ref{fig:2} and \ref{fig:1}, we compare this conjecture with numerics for 
$n=100$, $q=3$ and the 729 different $h$ of degree $<6$,
enumerated as $h_1=0,h_2=1,\dots,h_{729}=2T^5+\dots+2$ as indicated above.
We have sampled $N=5n^2=50000$ monic $f\in \FF_q[T]$ of degree $n$ chosen at random
and compute
\[
N_1(h_i) = \frac{\sum b_q(f)b_q(f+h_i)}{\sum b_q(f)b_q(f+1)} \qquad \mbox{and} \qquad H_1(h_i) = \frac{\frak{S}_{q,(0,h_i)}}{\mathfrak{S}_{q,(0,1)}},
\]
where the sum in $N_1$ is taken over the sampled $f$'s
and $H_1(h_i)$ is computed precisely. 
Figure~\ref{fig:1} is analogous to \cite[Figure 2]{ConnorsKeating}.
\begin{figure}[h]
\includegraphics[width=0.45\textwidth]{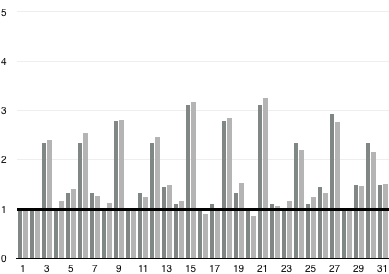}  
\hfil
\includegraphics[width=0.45\textwidth]{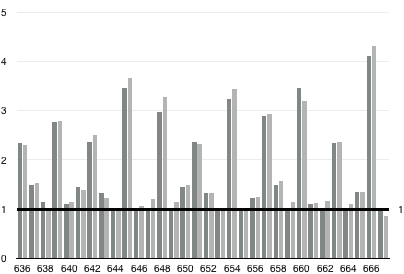}
\caption{\label{fig:2}
The plots show $N_1(h_i)$ (in light grey) and $H_1(h_i)$ (in dark grey) as functions of $i$ for the first $31$ indices (in the left plot) and for a random starting point (in the right plot) }
\end{figure}

\begin{figure}[h]
\includegraphics[width=0.49\textwidth]{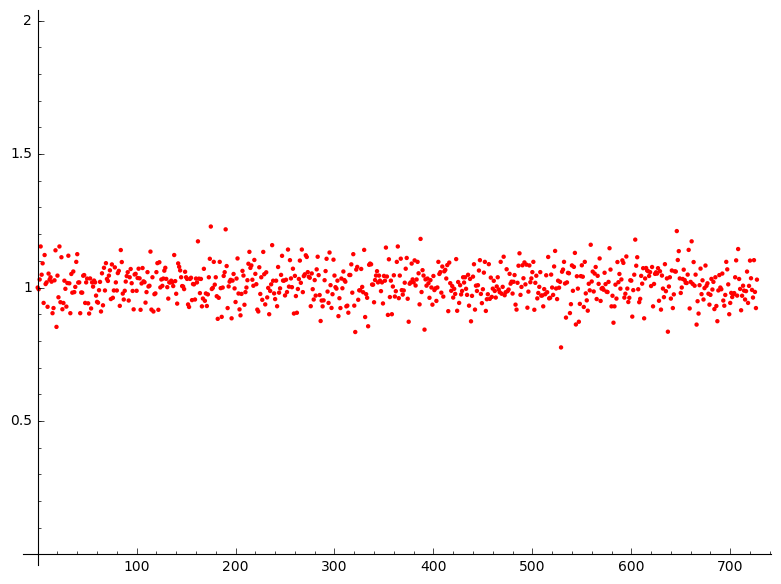}
\caption{\label{fig:1}
 $N_1(h_i)/H_1(h_i)$ as a function of $1\leq i\leq 729$ }
\end{figure}

\subsubsection{Varying $n$}
For $h=(h_1,\ldots, h_k)\in \FF_q[T]^k$ a $k$-tuple of pairwise distinct polynomials, Conjecture \ref{conj} predicts that
\begin{equation}\label{eq:varyinghconj}
\frac{\left<\prod_{i}b_q(f+h_i)\right>_{f\in M_{n,q}}}{\left<b_q(f)\right>_{f\in M_{n,q}}^k} \rightarrow \mathfrak{S}_{q,h},\quad n\rightarrow\infty.
\end{equation}
We compare this conjecture with numerics for $q=5$, $3\leq n<100$ and $h \in\{(0,1),(0,T)\}$ in Figure~\ref{fig:varyingn}, and $h=(0,1,T)$  in Figure~\ref{fig:varyingn-3correlation}. 
For each $n$ we have sampled $N=q n^{5/2}$ monic $f\in \FF_q[T]$ of degree $n$ chosen at random,
and for each $h=(h_1,\ldots, h_k)$ as above we computed
\[
N_{2,h}(n) = \frac{N^{k-1} \sum \prod_{i=1}^k b_q(f+h_i) }{(\sum b_q(f))^k}, 
\]
where the sums are taken over the sampled $f$'s. 

\begin{figure}[h]
\includegraphics[width = 0.49\textwidth]{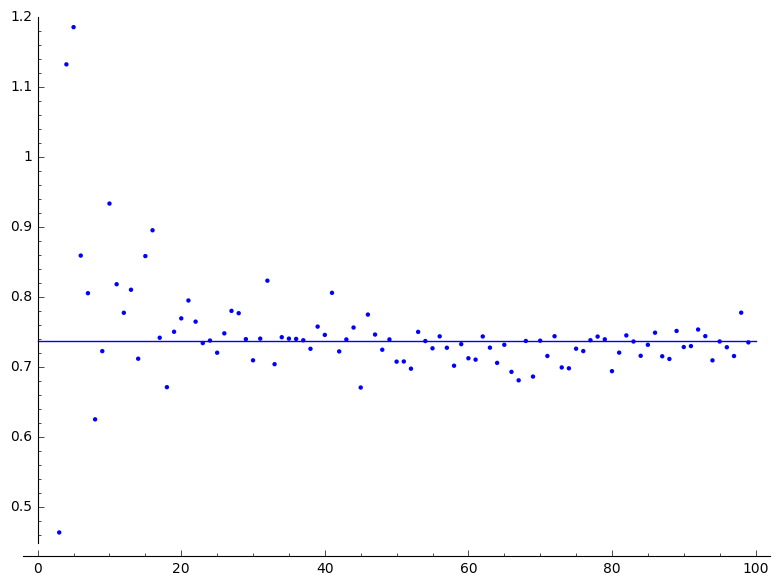}
\hfil
\includegraphics[width = 0.49\textwidth]{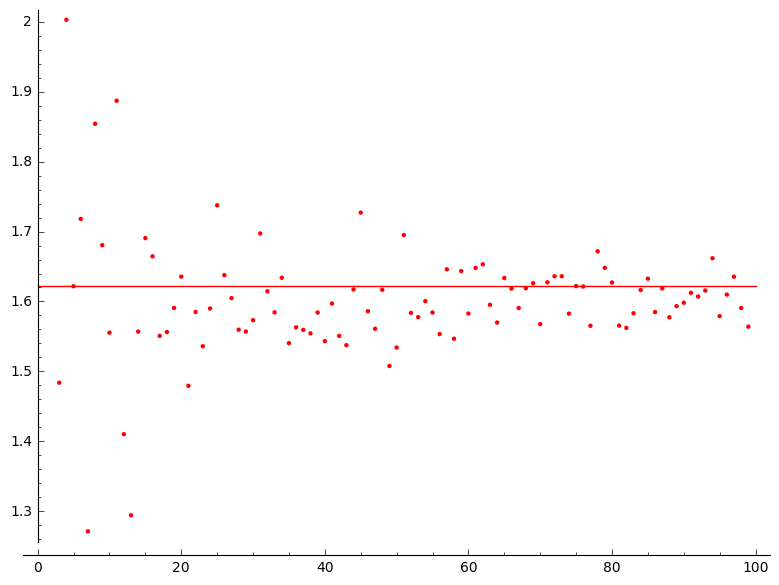}
\caption{\label{fig:varyingn}The plots of $N_{2,(0,1)}(n)$ (on the left) and $N_{2,(0,T)}$ (on the right) as functions of $n$. The lines are approximated values of the respective singular series $\mathfrak{S}_{5,(0,1)}$ and $\mathfrak{S}_{5,(0,T)}$.}
\end{figure}

\begin{figure}[h]
\includegraphics[width=0.49\textwidth]{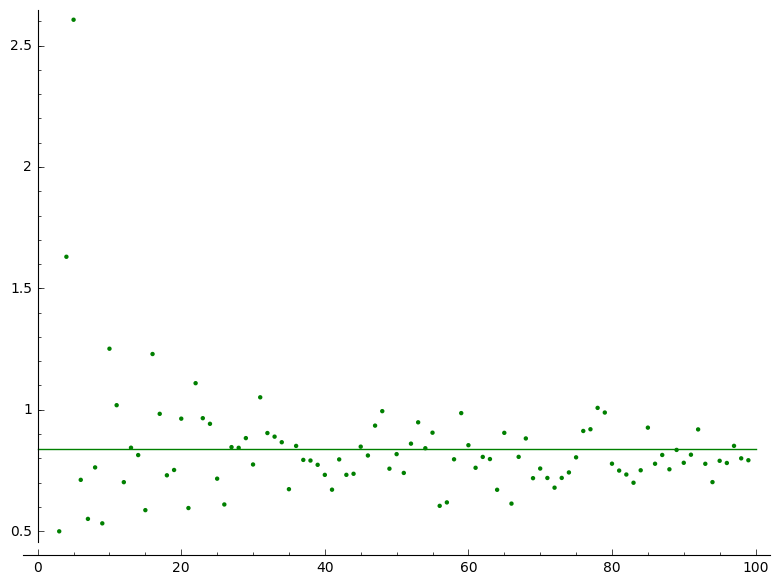}
\caption{\label{fig:varyingn-3correlation}$N_{2,(0,1,T)}(n)$ as a function of $n$. The line is an approximated value of the respective singular series $\mathfrak{S}_{5,(0,1,T)}$.}
\end{figure}

\subsubsection{Varying q}
Conjecture \ref{conj} (cf.\ \eqref{eq:1T}) suggests  that
\begin{equation}\label{eq:varyingqconj}
\frac{\left<b_q(f)b_q(f+T)\right>_{f\in M_{n,q}}}{\left<b_q(f)b_q(f+1)\right>_{f\in M_{n,q}}} = 2+\frac{1}{q}(1 + o(1 )), \qquad q\to \infty.
\end{equation}
In Figure~\ref{fig:varyingq} we compare this conjecture with numerics. For each prime $2<q<30$ we evaluate the left hand side of \eqref{eq:varyingqconj} by sampling random monic $f \in \FF_q[T]$ of degrees $3$ up to $49$ (and for $q=29$ up to $69$), and this we denote by $N_3(q)$. 
We compare it to $H_3(q)=2+1/q$.
\begin{figure}[h]
\includegraphics[width=0.49\textwidth]{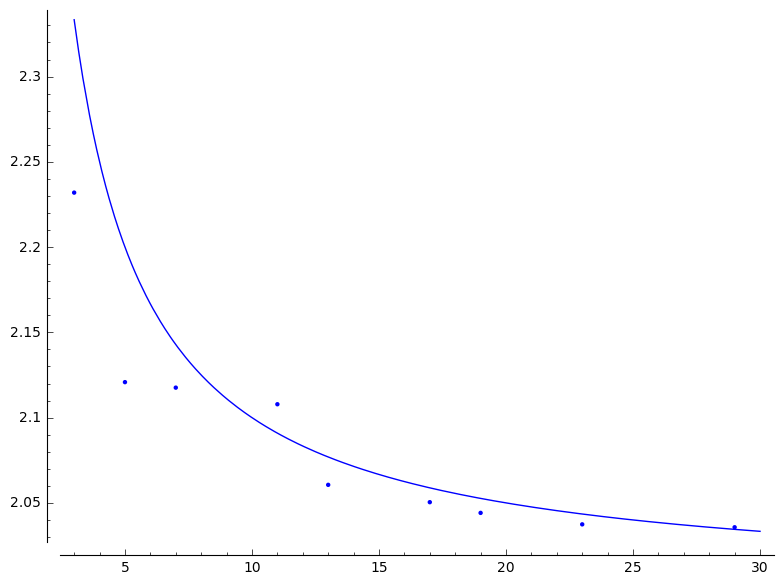}
\caption{\label{fig:varyingq}The curve is $H_3(q)=2+\frac{1}{q}$ as a function of $q$ and the points are $(q,N_3(q))$ for primes $2<q<30$.}
\end{figure}

\begin{figure}[h]
\includegraphics[width=0.45\textwidth]{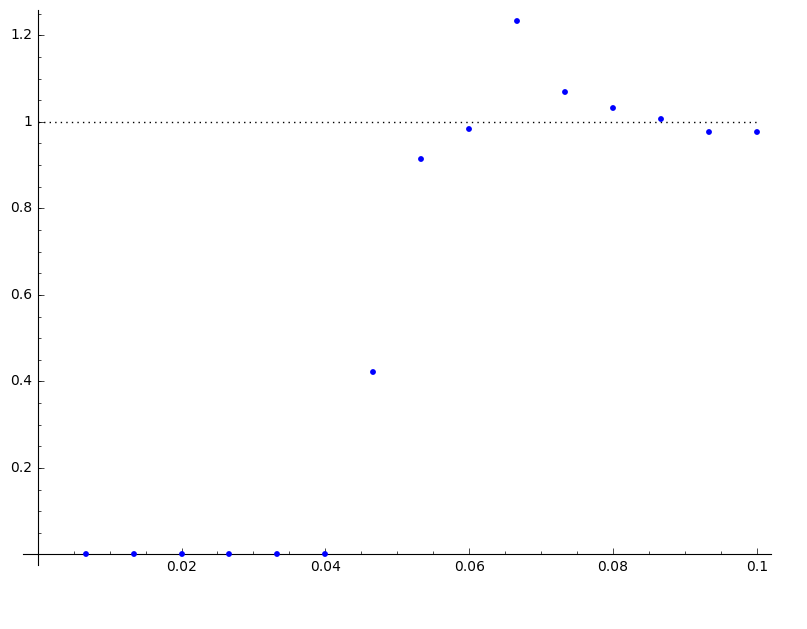}
\includegraphics[width=0.45\textwidth]{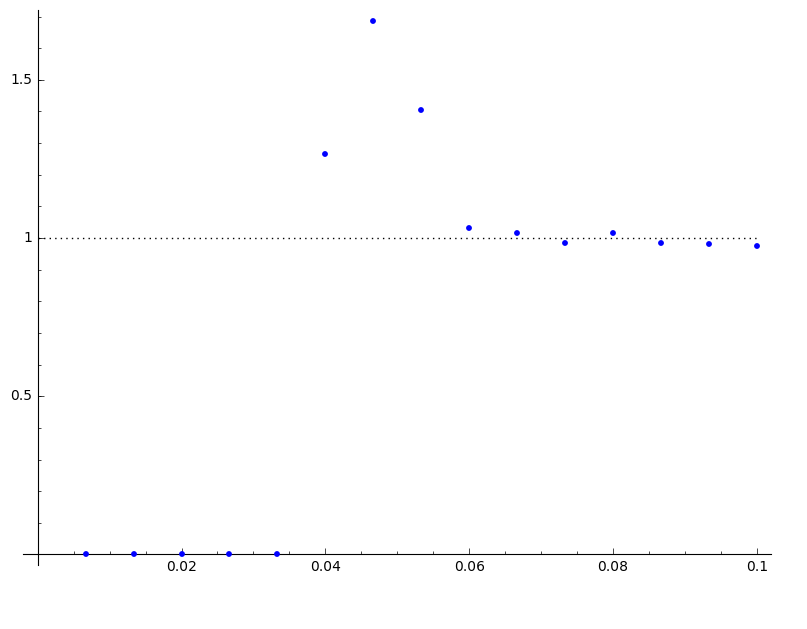}
\caption{\label{fig:varyingepsilon}
Here we plot $N_{4,f_0}(\epsilon)/H_{4,{\rm deg}(f_0)}$ as a function of $\epsilon$,
for $f_0(T)=T^{150}$ (in the left diagram) and a $f_0$ of degree $150$ that is randomly chosen
(in the right diagram).}
\end{figure}

\begin{figure}[h]
\includegraphics[width=0.45\textwidth]{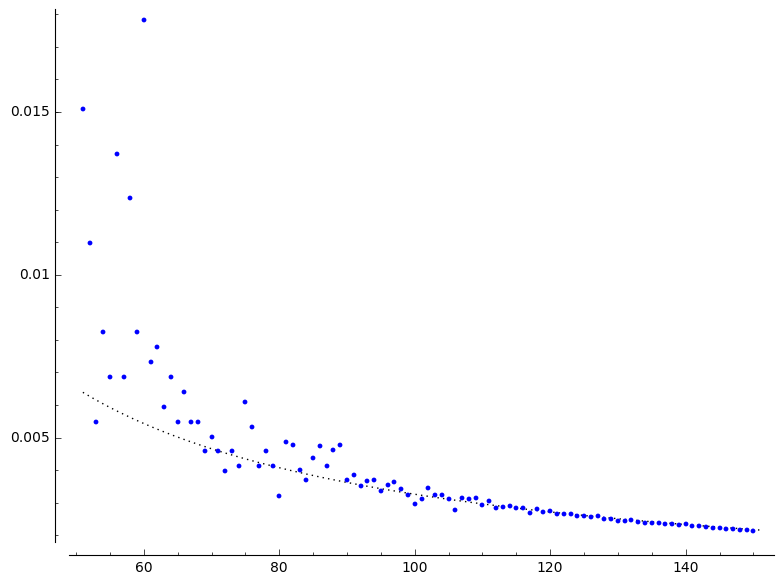}
\includegraphics[width=0.45\textwidth]{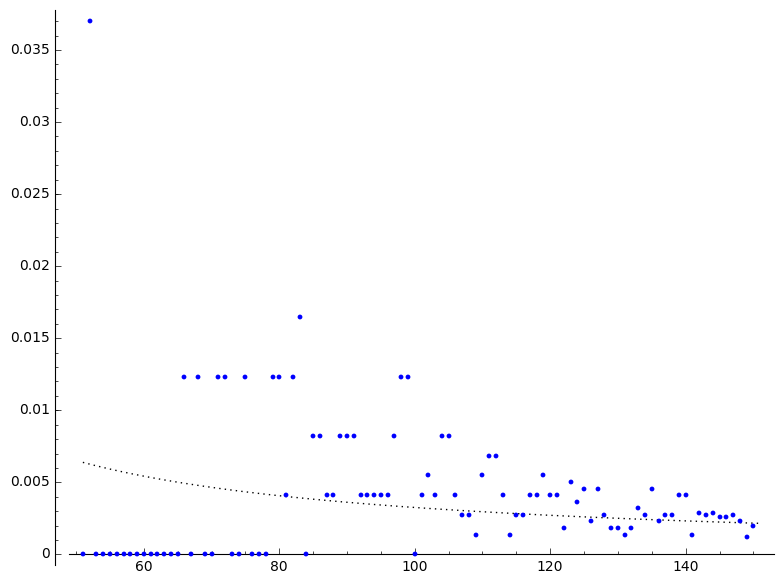}
\caption{\label{fig:varyingepsilon2}
Here we plot $N_{4,T^n}(\epsilon)$ as a function of $n$,
for $\epsilon=0.1$ (in the left diagram) and $\epsilon=0.05$ (in the right diagram),
and compare it to $H_{4,n}\approx 0.3257/n$. }
\end{figure}

\subsubsection{Short intervals}
We finally check Conjecture~\ref{conj_short} about autocorrelations at $h=(0,1)$ in short intervals.
Here we take $q=3$ and compute 
\begin{equation}\label{eq:average}
N_{4,f_0}(\epsilon)=\left< b_q(f)b_q(f+1) \right>_{|f-f_0|< |f_0|^{\epsilon}} 
\end{equation}
deterministically for fixed $f_0$ of degree $n\leq 150$ and various $\epsilon\leq 0.1$.
Note that $N_{4,f_0}(\epsilon)$ is piecewise constant, namely constant on each interval $\frac{i}{n}\leq\epsilon<\frac{i+1}{n}$,
and we chose one epsilon in each of these intervals, namely $\epsilon=\frac{i}{n}$.
In Figure~\ref{fig:varyingepsilon} we take two $f_0$
(namely $f_0=T^{150}$ and a $f_0$ of the same degree that is randomly chosen\footnote{$f_0=T^{150} + 2 T^{149} + T^{146} + T^{144} + 2 T^{143} + T^{140} + T^{138}
+ T^{137} + T^{136} + 2 T^{135} + 2 T^{133} + T^{132} + T^{130} + 2
T^{128} + 2 T^{127} + T^{125} + 2 T^{124} + 2 T^{122} + 2 T^{118} +
T^{117} + T^{116} + 2 T^{115} + T^{113} + T^{112} + 2 T^{111} + 2
T^{110} + T^{109} + T^{106} + 2 T^{103} + 2 T^{102} + 2 T^{100} + T^{98}
+ T^{97} + 2 T^{96} + T^{95} + 2 T^{94} + 2 T^{91} + 2 T^{90} + T^{88} +
2 T^{87} + 2 T^{86} + 2 T^{85} + T^{84} + 2 T^{82} + T^{81} + T^{80} + 2
T^{76} + T^{75} + T^{74} + T^{73} + 2 T^{72} + 2 T^{70} + T^{69} + 2
T^{67} + T^{66} + 2 T^{63} + T^{61} + T^{59} + 2 T^{58} + T^{57} + 2
T^{56} + T^{52} + T^{51} + 2 T^{50} + 2 T^{49} + T^{48} + 2 T^{47} +
T^{46} + 2 T^{44} + 2 T^{43} + 2 T^{40} + 2 T^{39} + T^{35} + T^{34} + 2
T^{33} + T^{32} + 2 T^{30} + 2 T^{29} + 2 T^{27} + 2 T^{26} + T^{25} +
T^{24} + T^{22} + T^{20} + T^{18} + 2 T^{16} + 2 T^{15} + 2 T^{14} + 2
T^{13} + 2 T^{12} + 2 T^{11} + T^{10} + 2 T^{9} + T^{8} + 2 T^{7} + 2
T^{6} + 2 T^{3} + 2 T^{2} + 2 T + 2$}), and we vary $\epsilon$.
In Figure~\ref{fig:varyingepsilon2} we take $\epsilon=0.1$ and $\epsilon=0.05$, and we vary $n$.
In each case, we compare $N_{4,f_0}(\epsilon)$ to the value 
$$
 H_{4,n} = \mathfrak{S}_{3,(0,1)}\cdot\left( K_3\cdot\frac{1}{4^n}\left(2n\atop n\right) \right)^2 \approx \frac{\mathfrak{S}_{3,(0,1)}K_3^2}{\pi n}
$$
predicted by Conjecture \ref{conj_short},
where we use approximate values
$\mathfrak{S}_{3,(0,1)}\approx 0.5736$
and
$K_3\approx 1.3357$.
The observed quantization in the right diagram of Figure 7 is explained by the fact that 
here for small $n$ the short interval contains very few polynomials -- 
for example for $60\leq n<80$ just enough to contain either $0$ or $1$ element $f$ with $b_q(f)b_q(f+1)=1$.

\section{Arithmetic functions of signed factorization type}
\label{sec:signed}

In this section we define arithmetic functions depending on signed factorization type (see Section~\ref{sec:signed_cycle_structure})
and prove a general result (Theorem \ref{thm:general}) for correlations of such functions in short intervals.
This result explains the autocorrelations through distributions on finite groups. 
More precisely, it allows to compute the large finite field limit of an expression of the form
$$
 \left<\prod_{i=1}^k\psi_i(f+h_i)\right>_{|f-f_0|<|f_0|^\epsilon}
$$
as an average over corresponding functions on a certain finite group
depending only on the combinatorics of $h_1(0),\dots,h_k(0)$.
In Section \ref{sec:signed_cycle_structure} we introduce arithmetic functions depending on signed factorization type,
in Section \ref{sec:functions_on_groups} we explain the corresponding functions on finite groups
and prove a general statement using a Chebotarev theorem.
In Sections \ref{sec:fiber_products} and \ref{sec:Galois_group} we then study the finite groups
that are relevant for correlations in short intervals,
namely fiber products of hyperoctahedral groups,
and in Section \ref{sec:proofgeneral} we state and prove the general theorem.

\subsection{Signed factorization type}
\label{sec:signed_cycle_structure}
A signed factorization type is a function 
\[
 \lambda\colon\mathbb{N}\times\mathbb{N}\times\{\pm1,0\}\rightarrow\mathbb{Z}_{\geq0}
\]
with finite support, and we denote by $\Lambda$ the set of all signed factorization types.
For $\lambda\in\Lambda$ we let
$$
 {\rm deg}(\lambda) = \sum_{d\in\mathbb{N}}\sum_{e\in\mathbb{N}}\sum_{s\in\{\pm1,0\}}\lambda(d,e,s)de
$$
and
\begin{eqnarray}
\label{eqn:def_chi}
 \chi(\lambda)&=&\begin{cases}(-1)^{\sum_{d,e\in\mathbb{N}}\lambda(d,e,-1)e}, &\mbox{ if }\lambda(d,e,0)=0\mbox{ for all }d,e\\
  0,&\mbox{ otherwise.}\end{cases}
\end{eqnarray}
To each $f\in M_q$ with prime factorization $f=P_1^{e_1}\cdots P_r^{e_r}$ 
we assign a signed factorization type by setting
$$
 \lambda_{f}(d,e,s) := \#\{i : {\rm deg}(P_i)=d, e_i=e, \chi_q(P_i)=s \}.
$$
Note that 
\[
\begin{array}{lll}
\deg(\lambda_f)&=&\deg(f),  \\
\chi(\lambda_f)&=&\chi_q(f),     \\
f \mbox{ is square-free} &\Leftrightarrow& \sum_{d,s}\sum_{e>1}\lambda_f(d,e,s)=0 
,\\
f(0)\neq 0 &\Leftrightarrow& \sum_{d,e}\lambda_f(d,e,0)=0.
\end{array}
\]
We denote by $\Lambda^*$ the space of functions $\psi\colon\Lambda\rightarrow\mathbb{C}$.
Each function $\psi\in\Lambda^*$ induces a family of arithmetic functions  $\psi_q$ (for $q$ an odd prime power)
on $M_q$ given by
\[
 \psi_q\colon M_q\rightarrow\mathbb{C},\quad f\mapsto\psi(\lambda_f).
\]
Many arithmetic functions are of this form, in particular:
\begin{enumerate}
\item The restriction of the quadratic character $\chi_q$ (see (\ref{eqn:chi_q})) to $M_q$, which is induced from
$\chi$ defined in (\ref{eqn:def_chi}).
\item The indicator function of prime polynomials, which is induced from
$$
 1_\mathbb{P}(\lambda)=\begin{cases}1,&\mbox{if }\sum_{d,s}\lambda(d,1,s)=\sum_{d,e,s}\lambda(d,e,s)=1\\0,&\mbox{otherwise}\end{cases},
$$
and the closely related
function field analogue of the von Mangoldt function,
see e.g.~\cite{KeatingRudnick},
which is induced from
$$
 \Lambda(\lambda)=\begin{cases} d_0, & \mbox{if }\sum_{e,s}\lambda(d_0,e,s)=\sum_{d,e,s}\lambda(d,e,s)=1,\\
 0, &\mbox{otherwise}, \end{cases}
$$
\item The function $b_q$ defined in (\ref{eqn:def_bq}), which by Proposition \ref{prop:Fermat} is induced from 
$$
 b(\lambda)=\begin{cases} 1, &\mbox{if }\lambda(d,2e+1,-1)=0\mbox{ for all }d,e, \\ 0, &\mbox{otherwise.} \end{cases}
$$
\item The function field analogue of the M\"obius function $\mu$, see e.g.~\cite{CarmonRudnick}.
\item The function field analogue of the function $r$ counting
the number of representations as sums of two squares, see Section \ref{sec:r}.
\item The function field analogue of the $r$-divisor function $d_r$, see Section \ref{sec:dkchi}.
\end{enumerate}




\subsection{Arithmetic functions on hyperoctahedral groups}
\label{sec:functions_on_groups}

The {\em hyperoctahedral group} (aka, \emph{signed permutation group}) of degree $n$ is the permutational wreath product 
\begin{equation}\label{hyperoctahedralgroup}
 {\rm H}_n=\mathbb{F}_2\wr S_n=V\rtimes S_n
\end{equation}
with $V=(\mathbb{F}_2)^n$.
The $S_n$-invariant subspaces of $V$ are 
\begin{equation}
\label{invariantsubgroups}
\begin{array}{lclp{.08\textwidth}lcl}
V_n&=&V, 
	&& V_{n-1}&=&\{x\in V:\sum_{i=1}^nx_i=0\},\\
V_1&=&\{(0,\dots,0),(1,\dots,1)\}, 
	&& V_0&=&\{(0,\dots,0)\},
\end{array}
\end{equation}
cf.~\cite[Lemma 4.2]{BBF}.
In particular, $V_{n-1}\rtimes S_n$, $V_n\rtimes A_n$ and $V_{n-1}\rtimes A_n$ are normal subgroups of 
$\mathrm{H}_n$. 
The \emph{total sign} homomorphism
\begin{equation}\label{eq:ts}
\ts_n\colon \mathrm{H}_n \to \{\pm1\}
\end{equation}
is given by $\ts_n(x\tau) = (-1)^{\sum_{i} x_i}$, where $x=(x_1,\ldots, x_n)\in V$ and $\tau\in S_n$. It is obvious that $\ts_n$ is surjective and $\ker \ts_n = V_{n-1}\rtimes S_n$. 

To each $\sigma=x\tau\in \mathrm{H}_n$, with $x\in V$ and $\tau\in S_n$, we assign a signed factorization type by setting
\begin{equation}
\label{def:lambdasig}
\lambda_\sigma(d,e,s) = 
\begin{cases}
 	\#\{\Omega:\Omega\mbox{ is an orbit of }\tau,\;\#\Omega=d,\;(-1)^{\sum_{i\in\Omega}x_i}=s\}, &\mbox{if }e=1\\
                      0,&\mbox{otherwise}
\end{cases}
\end{equation}
Note that 
\begin{eqnarray*}
 \deg(\lambda_\sigma)&=&n, \\
 \chi(\lambda_\sigma)&=&\ts_n(\sigma).
\end{eqnarray*} 
Each $\psi\in\Lambda^*$ induces a family of maps 
\[
\psi_n\colon \mathrm{H}_n\rightarrow\mathbb{C},\quad g\mapsto\psi(\lambda_g).
\]
In particular, the function $\chi$ defined in \eqref{eqn:def_chi} induces the total sign maps $\chi_n$. 
From now on, we abuse notation and write $\chi$ instead of $\chi_n$.

We denote by $\mathrm{H}_n^k$ the direct product of $k$ copies of $\mathrm{H}_n$
and by $\pi_i\colon \mathrm{H}_n^k\rightarrow\mathrm{H}_n$ the projection onto the $i$-th component.
For a subgroup $G$ of $\mathrm{H}_n^k$ we define a $k$-multilinear map $\left<\cdot\right>_G\colon(\Lambda^{*})^k\rightarrow\mathbb{C}$ by
\[
 \left<\psi_1,\dots,\psi_k\right>_G := \left<\prod_{i=1}^k\psi_{i,n}(\pi_i(\sigma))\right>_{\sigma\in G} = \frac{1}{\#G}\cdot\sum_{\sigma\in G}\prod_{i=1}^k\psi_i(\lambda_{\pi_i(\sigma)}).
\]
For example, if $G=\mathrm{H}_n^k$, we have \emph{independence}; namely,  
\[
\left<\psi_1,\dots,\psi_k\right>_G = \prod_{i=1}^k \left< \psi_i\right>_{\mathrm{H}_{n}}.
\]
 
Now we describe the connection with arithmetic. Let $K$ be a field of characteristic different from $2$ and let $f\in K[T]$ of degree $n$.
Suppose that $f$ is separable and $f(0)\neq 0$. 
If $\omega_1,\dots,\omega_n$ is an enumeration of the roots of $f$, then
the roots of $f(-T^2)$ are $\pm \rho_1,\dots,\pm\rho_n$,
where $\rho_i$ is a fixed square root of $-\omega_i$.
The Galois action then induces embeddings ${\rm Gal}(f(T)|K)\rightarrow S_n$ and 
\begin{equation}\label{Theta}
 \Theta\colon {\rm Gal}(f(-T^2)|K)\rightarrow\mathrm{H}_n=V\rtimes S_n,
\end{equation}
where the elements of $S_n$ act by permuting the $\omega_i$, 
and the elements of $V$ act by sign change on the $\rho_i$, 
cf.~\cite[Lemma 3.1]{BBF}.
Over the finite  $K=\mathbb{F}_q$, the Galois group in \eqref{Theta} is generated by the Frobenius automorphism $\phi_q\colon x\mapsto x^q$. 
In fact $\lambda_f$ and $\lambda_{\Theta(\phi_q)}$ are equal: 

\begin{lemma}
\label{lem:lambda_f_Theta}
Let $f\in \FF_q[T]$ monic, square-free and not divisible by $T$. Then $\lambda_f = \lambda_{\Theta(\phi_q)}$, where $\phi_q$ is the Frobenius automorphism and $\Theta$ is the map from \eqref{Theta}
\end{lemma}

\begin{proof}
Let $f=P_1^{e_1}\cdots P_r^{e_r}$ be the prime factorization of $f$.
Note that by assumption, $\chi_q(P_i)\in\{\pm1\}$ and $e_i=1$ for all $i$,
hence
$\lambda_f(d,e,s)=0=\lambda_{\Theta(\phi)}(d,e,s)$ for all $(d,e,s)$ with either $e>1$ or $s=0$.
The set of roots $\Omega=\{\omega_1,\dots,\omega_n\}$ of $f$
is partitioned as $\Omega=\coprod_{i=1}^r\Omega_i$, where $\Omega_i$ is the set of roots of $P_i$.
Write $\Theta(\phi_q)=x\tau$, 
and note that the $\Omega_i$ are exactly the orbits of $\tau$,
of length $\#\Omega_i= \deg(P_i)$.
By \cite[Lemma 3.3]{BBF}, $\chi_q(P_i)=(-1)^{\sum_{j\in\Omega_i} x_j}$.
It follows that for any $(d,s)$ with $s\neq 0$,
\begin{eqnarray*}
 \lambda_f(d,1,s) &=& \#\{i : {\rm deg}(P_i)=d, e_i=1, \chi_q(P_i)=s\} \\
  &=& \#\{i : \#\Omega_i=d, (-1)^{\sum_{j\in\Omega_i} x_j}=s \}
  = \lambda_{\Theta(\phi)}(d,1,s).
\end{eqnarray*}
Thus, $\lambda_f(d,e,s)=\lambda_{\Theta(\phi)}(d,e,s)$ for all $(d,e,s)$.
\end{proof}

Given $k$ polynomials $f_1,\dots,f_k\in K[T]$ that are of degree $n$, separable, with $f_i(0)\neq 0$,
and pairwise coprime, 
the Galois group ${\rm Gal}(\prod_{i=1}^k f_i(-T^2)|K)$ embeds into $\mathrm{H}_n^k$.
The following proposition computes the limit of correlations of arithmetic functions depending on signed factorization type,
where the correlation can be taken over any set of polynomials in $\mathbb{F}_q[T]$
that are specializations of a fixed polynomial $f_A\in\mathbb{F}_q[A_1,\dots,A_m][T]$,
where by a specialization of $f_A$ we mean the polynomial $f_a\in\mathbb{F}_q[T]$ obtained
by substituting $a_1,\dots,a_m\in\mathbb{F}_q$ for the variables $A_1,\dots,A_m$.
For example, the short interval $|f-f_0|<q^m$ is the set of specializations of
$f_A=f_0+\sum_{i=0}^{m-1} A_iT^i$.

\begin{proposition}\label{prop:groupcorrelation}
Fix $n\geq 1$, $m\geq 0$, $k\geq 1$ and $\psi_1,\dots,\psi_k\in\Lambda^*$.
Let $q$ be an odd prime power,
let $f_{1,A},\dots,f_{k,A}\in\mathbb{F}_q[A_1,\dots,A_m][T]$
be monic, of degree $n$, square-free, not divisible by $T$, and pairwise coprime,
and denote 
$$
 g_A(T)=\prod_{i=1}^kf_{i,A}(-T^2)
$$ 
and 
$G={\rm Gal}(g_A|\mathbb{F}_q(A))\leq\mathrm{H}_n^k$.
Assume that $G={\rm Gal}(g_A|\overline{\mathbb{F}}_q(A))$.
Then
$$
 \left< \prod_{i=1}^k\psi_{i,q}(f_{i,a})\right>_{a\in\mathbb{F}_q^m} = 
 \left<\psi_1,\dots,\psi_k\right>_G + O(q^{-1/2}),
$$
where the implied constant depends only on $n$, $m$, $k$ and $\psi_1,\dots,\psi_k$.
\end{proposition}

\begin{proof}
For $\underline{\lambda}=(\lambda_1,\dots,\lambda_k)\in\Lambda^k$ let
$$
 A(\underline{\lambda}) = \{a\in\mathbb{F}_q^m : (\lambda_{f_{1,a}},\dots,\lambda_{f_{k,a}})=\underline{\lambda}\}
$$
and 
$$
 G(\underline{\lambda}) = \{\sigma\in G : (\lambda_{\pi_1(\sigma)},\dots,\lambda_{\pi_k(\sigma)})=\underline{\lambda}\},
$$
and note that 
$$
 \prod_{i=1}^k\psi_{i,q}(f_{i,a})=\prod_{i=1}^k\psi_i(\lambda_i)=\prod_{i=1}^k\psi_{i,n}(\pi_i(\sigma))
$$ 
for all $a\in A(\underline{\lambda})$ and $\sigma\in G(\underline{\lambda})$.
Thus, since each $\psi\in \Lambda^*$ takes only a finite number of values on the set of $\lambda\in\Lambda$ with ${\rm deg}(\lambda)=n$,
it suffices to prove for each $\underline{\lambda}$ that
$$
 \frac{\#A(\underline{\lambda})}{q^m } = \frac{\#G(\underline{\lambda})}{\#G} + O(q^{-1/2}).
$$
Note that the assumption on the $f_{i,A}$'s can be rephrased as saying that the discriminant $\Delta\in\mathbb{F}_q[A]$ 
of $g_A$ is non-zero.
Also note that $\Delta(a)$ is the discriminant of $g_a(T)=\prod_{i=1}^kf_{i,a}(-T^2)$.
We distinguish two cases:

If $\lambda_i(d,e,s)>0$ for some $i$ and either $e>1$ or $s=0$,
then $\lambda_{\pi_i(\sigma)}$ equals $\lambda_i$ for no $\sigma$, so $\#G(\underline{\lambda})=0$.
On the other side,
$\lambda_{f_{i,a}}=\lambda_i$ only if $\Delta(a)=0$,
so as ${\rm deg}(\Delta)$ is bounded only in terms of $n$ and $k$,
we get $\#A(\underline{\lambda})=O(q^{m-1})$,
see e.g.~\cite[Ch.~4 Lemma 3A]{Schmidt}.

If $\lambda_i(d,e,s)=0$ for all $i$ and all $(d,e,s)$ with $e>1$ or $s=0$, 
then we have embeddings $\Theta_{i,a}\colon{\rm Gal}(f_{i,a}|\mathbb{F}_q)\rightarrow\mathrm{H}_n$
and $\Theta_a\colon {\rm Gal}(\prod_if_{i,a}|\mathbb{F}_q)\rightarrow\mathrm{H}_n^k$,
and both Galois groups are generated by the Frobenius automorphism $\phi_q$.
Applying the Chebotarev theorem \cite[Theorem 5.1]{BBF} to the polynomial $g_A$ and 
the $G$-invariant set $G(\underline{\lambda})$ gives that
\begin{eqnarray}\label{eqn:Chebotarev}
 \#\{a\in\mathbb{F}_q^m : \Delta(a)\neq0\mbox{ and }\Theta_a(\phi_q)\in G(\underline{\lambda})\} = \frac{\#G(\underline{\lambda})}{\#G}≈Ω\cdot q^m + O(q^{m-1/2}).
\end{eqnarray}
Since $\Theta_{i,a}(\phi_q)=\pi_i(\Theta_a(\phi_q))$,
and 
$\lambda_{f_{i,a}}=\lambda_{\Theta_{i,a}(\phi_q)}$
by Lemma \ref{lem:lambda_f_Theta},
we get that the left hand side of $(\ref{eqn:Chebotarev})$ equals precisely 
$A(\underline{\lambda})$,
which gives the claim.
\end{proof}

Next we provide some group theoretical tools to determine $G$ in the case $f_{i,A}=f_0+h_i+\sum_{j=0}^mA_jT^j$,
needed to apply Proposition \ref{prop:groupcorrelation} for correlations in short intervals.

\subsection{Fiber products of hyperoctahedral groups}
\label{sec:fiber_products}

For a finite group $G$ we denote by $M(G)$ the intersection over the maximal normal subgroups of $G$. In the literature, $M(G)$ is sometimes called the \emph{Melnikov subgroup} of $G$.
It is the smallest normal subgroup of $G$ with quotient a direct product of simple groups.
We will use several times that subgroups of finite elementary abelian $p$-groups are again elementary abelian $p$-groups
and therefore products of simple groups.
We will also use that if $N\unlhd G$, then $M(N)\subseteq M(G)$, and if $f\colon G\rightarrow H$ is a homomorphism, then 
$f(M(G))= M(f(G))$,
cf.~\cite[Lemma 25.5.4]{FJ}.

\begin{lemma}\label{lem:Mproduct}
Let $(G_i)_{i\in I}$ be a family of finite groups and let $G:=\prod_{i\in I}G_i$.
Then $M(G)=\prod_{i\in I}M(G_i)$.
\end{lemma}

\begin{proof}
For each $i$, $G_i\unlhd G$ implies $M(G_i)\subseteq M(G)$, and therefore $ \prod_{i\in I}M(G_i)\subseteq M(G)$. The other inclusion follows as  $G/\prod_{i\in I}M(G_i)=\prod_{i\in I}G_i/M(G_i)$ is a product of simple groups, so $M(G)=\prod_{i\in I}M(G_i)$.
\end{proof}

\begin{proposition}\label{prop:fullproduct}
Let $(G_i)_{i\in I}$ be a finite family of finite groups, put $G = \prod_{i\in I} G_i$, let $H\leq G$, and let $\pi_i\colon G\rightarrow G_i$ be the projection map.
If $HM(G)=G$ and $\pi_i(H)=G_i$ for all $i$, then $H=G$.
\end{proposition}

\begin{proof}
Suppose first that $I=\{1,2\}$. Then by Goursat's lemma, there exist a group $Q$ and epimorphisms $f_i\colon G_i\rightarrow Q$ such 
that $H$ is the fiber product $G_1\times_Q G_2$. 
Let $M$ be a maximal normal subgroup of $Q$, put $S=Q/M$, and let $\pi\colon Q\rightarrow S$ be the quotient map. 
Then $H$ is contained in the fiber product $G_1\times_S G_2$. 
Since $M_i:={\rm ker}(\pi\circ f_i)$ is a maximal normal subgroup of $G_i$, 
we have 
$$
 M(G)=M(G_1)\times M(G_2)\subseteq M_1\times M_2\subseteq G_1\times_S G_2.
$$ 
Thus, 
$$
 HM(G)\subseteq G_1\times_S G_2\subsetneqq G_1\times G_2,
$$
contradicting the assumption $HM(G)=G$.

The general case now follows by induction: Indeed, for $J\subseteq I$, put $G_J = \prod_{i\in J} G_j$, let $\pi_J\colon G\rightarrow G_J$ be the projection map,
and let $H_J:=\pi_J(H)$.
Then $\pi_i(H_J)=\pi_i(H)=G_i$ for all $i\in J$,
and 
$$
 H_JM(G_J)=\pi_J(H) M(\pi_J(G))\supseteq \pi_J(H M(G))=\pi_J(G)=G_J,
$$ 
so if $J=I\setminus\{i_0\}$, then the induction hypothesis can be applied to $G_J=\prod_{i\in J}G_i$ and $H_J$,
showing that $H_J=G_J$,
and then the case $I=\{1,2\}$ can be applied to $G=G_J\times G_{i_0}$ and $H$,
showing that $H=G$.
\end{proof}

Next we compute the Melnikov subgroups of the hyperoctahedral group and of one of its subgroups. 
\begin{lemma}\label{lem:MG}
Under the notation of \S\ref{sec:functions_on_groups}, if $n\geq 3$, then
\[
M(V\rtimes S_n)=M(V_{n-1}\rtimes S_n)=V_{n-1}\rtimes A_n.
\]
\end{lemma}

\begin{proof}
Let $W$ be either $V_n$ or $V_{n-1}$. Write 
\[
H=W\rtimes S_n \quad\mbox{and}\quad K=V_{n-1}\rtimes A_n\leq H.
\]
It is immediate that $M(H)\subseteq K$,
as $K$ is either a maximal normal subgroup of $H$, if $W=V_{n-1}$ 
or an intersection of such, if $W=V$.
To show the converse inclusion, 
we let $N$ be a maximal normal subgroup of $H$
and aim to show that $K\subseteq N$.
Let $A:=H/N$ and $U:=N\cap W$. Then $A$ is a finite simple group and  $U$ is an $S_n$-invariant subspace of $V$ that is contained in $W$. 

If $N$ does not contain $W$, then $WN=H$ from maximality. 
Thus, $W/U \cong H/N= A$.
As $W$ is a $2$-elementary abelian group, this implies that $A\cong \mathbb{Z}/2\mathbb{Z}$ and that $U$ has index $2$ in $W$. 
By \eqref{invariantsubgroups}, since $n\geq 3$ this can happen only if $W=V$ and $U=V_{n-1}$. 
Hence, from maximality $N = V_{n-1}\rtimes S_n$ and thus $K\subseteq N$. 

If $N$ contains $W$, then $N/W$ is a maximal normal subgroup of $H/W\cong S_n$. 
Since $S_n$ has a unique maximal normal subgroup, 
namely $A_n$, we conclude that $N=W\rtimes  A_n$ and so $N$ contains $K$. 
%
%
%
\end{proof}

For a finite set $I$ we let $\mathrm{H}_n^{(I)}$ be the fiber product of copies of $\mathrm{H}_n$ with respect to $\ts$; i.e.,
$$
 \mathrm{H}_n^{(I)} := \left\{ (g_i)_{i\in I}\in\prod_{i\in I}\mathrm{H}_n : \ts(g_i)=\ts(g_j)\mbox{ for all }i,j\in I\right\},
$$
and we denote by $\pi_i\colon \mathrm{H}_n^{(I)}\rightarrow\mathrm{H}_n$ the $i$-projection map.
Note that
$$
 \prod_{i\in I} V_n\rtimes S_n\geq \mathrm{H}_n^{(I)} \geq \prod_{i\in I}V_{n-1}\rtimes S_n \geq \prod_{i\in I}V_{n-1}\rtimes A_n .
$$

\begin{lemma}\label{lem:MF}
For $n\geq 3$ and any finite set $I$,
$$
 M(\prod_{i\in I} V_n\rtimes S_n) = M(\mathrm{H}_n^{(I)}) = M(\prod_{i\in I}V_{n-1}\rtimes S_n) = \prod_{i\in I}V_{n-1}\rtimes A_n.
$$ 
\end{lemma}

\begin{proof}
Let 
\begin{equation}\label{eq:defining_groups}
\Gamma=\prod_{i\in I}V_n\rtimes S_n,\quad N=\prod_{i\in I}V_{n-1}\rtimes S_n, \quad \mbox{and}\quad K=\prod_{i\in I}V_{n-1}\rtimes A_n.
\end{equation}
As $N$ is normal in $G$  with abelian quotient and $\chi(x)=1$ for each $x\in V_{n-1} \rtimes S_n$, we have that $N\unlhd \mathrm{H}_n^{(I)}\unlhd \Gamma$. Hence $M(N)\leq M(\mathrm{H}_n^{(I)})\leq M(\Gamma)$.
Lemmas~\ref{lem:Mproduct} and \ref{lem:MG} give that $M(\Gamma)=M(N)=K$,
which thus proves the claim.
%
\end{proof}



We prove an analogue of Proposition~\ref{prop:fullproduct} for fiber products of $\mathrm{H}_n$.

\begin{proposition}\label{prop:fullfiberproduct}
Fix a finite set $I$, put $G=\mathrm{H}_n^{(I)}$, and let $H\leq G$.
If $H M(G)=G$ and $\pi_i(H)= \mathrm{H}_n$ for all $i\in I$, then $H=G$.
\end{proposition}

\begin{proof}
Let $N$ and $K$ be as in \eqref{eq:defining_groups}. 
%
By Lemma \ref{lem:MF}, $M(G)=M(N)=K$.
The assumption $HK=HM(G)=G$ implies an isomorphism  $G/K \cong H/H\cap K$. In particular, if we put $H_0=H\cap N$, then as $N$ is a normal in $G$, $H_0$ is normal in $H$ with $H/H_0\cong G/N\cong \mathbb{F}_2$ and we have $N=H_0K$.

Since 
$\pi_i(H)=\pi_i(G)$ and $[\pi_i(G):\pi_i(N)]=2$, we conclude that $\pi_i(H_0)=\pi_i(N)$.
Therefore, Proposition \ref{prop:fullproduct} gives that $H_0=N$,
from which we conclude that $H=G$.
\end{proof}

We shall need the following notation for later use: For a finite family $\mathcal{I}=(I_j)_{j\in J}$ of nonempty finite sets
we denote
$$
 \mathrm{H}_n^\mathcal{I} := \prod_{j\in J}\mathrm{H}_n^{(I_j)} \leq \prod_{j\in J}\prod_{i\in I_j}\mathrm{H}_n.
$$

\subsection{The Galois group of correlations in short intervals}
\label{sec:Galois_group}

Let $F$ be a field of characteristic different from $2$, let $n>m\geq 2$ and $r>0$, and let $K=F(A_0,\dots,A_m)$.
Let $f_0\in F[T]$ be monic of degree $n$ and $h_1,\dots,h_k\in F[T]$ of degree less than $n$ and pairwise distinct.
Starting from
$$
 f(T) = f_0(T) + \sum_{i=0}^mA_iT^i \in K[T],
$$
we define $f_i(T) = f(T)+h_i(T)$ and $f_h(T) = \prod_{i=1}^k f_i(T)$.
Note that
$$
 f_i(0) = f_0(0) + A_0 + h_i(0).
$$
The goal of this section is to determine the Galois group of $f_h(-T^2)$.

If $\omega_1,\dots,\omega_n$ is an enumeration of the roots of $f_i$, then
the roots of $f_i(-T^2)$ are $\pm \rho_1,\dots,\pm\rho_n$,
where $\rho_i$ is a fixed square root of $-\omega_i$.
The Galois action induces embeddings ${\rm Gal}(f_i(T)|K)\hookrightarrow S_n$ and $\Theta_i\colon {\rm Gal}(f_i(-T^2)|K)\hookrightarrow\mathrm{H}_n=V\rtimes S_n$
as in  \eqref{eq:defining_groups}.

\begin{proposition}\label{prop:isoms}
For each $i$, the action on the roots of $f_i(T)$  and $\Theta_i$ induce isomorphisms 
$$
 {\rm Gal}(f_i(T)|K)\cong S_n\quad\mbox{ and }\quad {\rm Gal}(f_i(-T^2)|K)\cong\mathrm{H}_n=V\rtimes S_n.
$$ 
Here, the extensions of $K$ corresponding to the subgroups $V\rtimes A_n$ and $V_{n-1}\rtimes S_n$ of $\mathrm{H}_n$ are
$K(\sqrt{{\rm discr}(f_i)})$ and $K(\sqrt{f_i(0)})$, respectively.
\end{proposition}

\begin{proof}
See \cite[Proposition 4.6]{BBF} for the claimed isomorphisms.
The subgroup $V\rtimes A_n$ of $V\rtimes S_n$ corresponds 
to the same extension as the subgroup $A_n$ of $S_n$, namely to $K(\sqrt{{\rm discr}(f_i)})$, cf.~\cite[Corollary 4.2]{Milne}.
Since $f_i(0)=(-1)^n\omega_1\cdots\omega_n=(\rho_1\cdots\rho_n)^2$, the splitting field of $f_i(-T^2)$ contains $K(\sqrt{f_i(0)})$.
Since the elements of $V_{n-1}$ act on the roots of $f_i(-T^2)$ by an even number of sign changes, 
$\sqrt{f_i(0)}=\rho_1\cdots\rho_n$ is invariant under $V_{n-1}\rtimes S_n$,
which implies that $K(\sqrt{f_i(0)})$ is the field extension of $K$ corresponding to the subgroup $V_{n-1}\rtimes S_n$ of $\mathrm{H}_n$.
\end{proof}

As before we have the maps 
\begin{equation*}
\xymatrix@1{
 \Psi_i\colon {\rm Gal}(f_h(-T^2)|K)\ar[r]^-{ {\rm res}_i}&
	{\rm Gal}(f_i(-T^2)|K)\ar[r]^-{\Theta_i}&
	\mathrm{H}_n}, \quad i=1,\ldots, k,
\end{equation*}
which induce an embedding 
\begin{equation}\label{Psi}
\Psi\colon {\rm Gal}(f_h(-T^2)|K)\to \mathrm{H}_n^k.
\end{equation}
Denote by $H$ the image of $\Psi$, $G=\prod_{i=1}^k\mathrm{H}_n$, and  $\pi_i\colon G\rightarrow\mathrm{H}_n$ the projection onto the $i$-th factor.  Proposition~\ref{prop:isoms} implies that 
\begin{equation} \label{pi(h)=pi(G)}
\pi_i(H)=\pi_i(G) = \mathrm{H}_n, \qquad \mbox{ for all $i=1,\ldots, k$.}
\end{equation}



Let 
\begin{align}
 J&=\{h_i(0):i=1,\dots,k\},\\
 I_a&=\{i:h_i(0)=a\},& \mbox{ for }a\in J,\\
 \mathcal{I}_h &= (I_a)_{a\in J}.
 \label{eqn:def_I_h}
\end{align}

Recall that a family $(x_1,\dots,x_s)$ of elements of $K^\times$ is {\em square-independent} if the residues in $K^\times/K^{\times 2}$ are $\mathbb{F}_2$-independent. By Kummer theory, this is equivalent to
$$
 {\rm Gal}(K(\sqrt{x_1},\dots,\sqrt{x_s})|K)\cong(\mathbb{F}_2)^s.
$$ 

\begin{lemma}\label{lem:squareindep}
The family 
$$
 (f_0(0)+A_0+a)_{a\in J}\cup({\rm discr}(f_i))_{i=1,\dots,k}
$$
is square-independent in $K$.
\end{lemma}

\begin{proof}
By \cite[Proposition 2.1]{BB}, the ${\rm discr}(f_i)$ are non-squares and pairwise coprime in the ring $R=F(A_1,\dots,A_m)[A_0]$.
For $a\in J$ and $i=1,\ldots, k$,  writing $f_i$ as $\tilde{f}_i+A$ with $A=f_0(0)+A_0+a$ and $\tilde{f}_i\in F(A_1,\dots,A_m)[T]$, then $\tilde{f}_i$ is  separable and 
\cite[Lemma~4.5]{BBF} gives that $A$ does not divide ${\rm discr}(f_i)$ in $R$.
Since $A$ is a prime element in the UFD $R$, for each $a\in J$, 
this shows that the family $(f_0(0)+A_0+a)_{a\in J}\cup({\rm discr}(f_i))_{i=1,\dots,r}$
consists of pairwise coprime non-squares in $R$,
which implies that it is square-independent in the fraction field $K$ of $R$.
\end{proof}


For each $i$, let $K_i=K(\sqrt{f_i(0)})=K(\sqrt{f_0(0)+A_0+h_i(0)})$ and let 
$$
 {\rm res}_{K_i}\colon {\rm Gal}(f_i(-T^2)|K)\rightarrow{\rm Gal}(K_i|K)
$$ 
denote the corresponding restriction map.

\begin{proposition}
\label{prop:Galois_group_short_interval}
The map $\Psi$ given in \eqref{Psi} induces an isomorphism
\[
 {\rm Gal}(f_h(-T^2)|K) 
\cong  \mathrm{H}_n^{\mathcal{I}_h},
\]
where $\mathcal{I}_h$ is defined in \eqref{eqn:def_I_h}.
\end{proposition}

\begin{proof}
Let $H$ be the image of ${\rm Gal}(f_h(-T^2)|K)$ in $G=\prod_{i=1}^k \mathrm{H}_n = \prod_{a\in J}\prod_{i\in I_a}\mathrm{H}_n$.
As $\Psi$ is injective, it suffices to show that  $H=\mathrm{H}_n^{\mathcal{I}_h}$.

We first treat the case $J=\{a\}$.
Since $K_i$ is the same for all $i\in I_a$, we call it $K_a$.
It is a common subfield of the splitting fields of all $f_i(-T^2)$,
hence $H$ is contained in $\mathrm{H}_n^{(I_a)}$ 
We apply Proposition~\ref{prop:fullfiberproduct}, where  $\pi_i(H)=\mathrm{H}_n$ by \eqref{pi(h)=pi(G)}.
By Lemma \ref{lem:MF}, 
$$
 M(\mathrm{H}_n^{(I_a)})=\prod_{i\in I_a}V_{n-1}\rtimes A_n,
$$ 
and $M(\mathrm{H}_n^{(I_a)})\cap H$
corresponds to the field extension 
$$
 K_a\cdot K(\sqrt{{\rm discr}(f_i)},{i\in I_a}).
$$ 
By Lemma \ref{lem:squareindep}, $f_0(0)+A_0+a$ and $({\rm discr}(f_i))_{i\in I_a}$ are square-independent,
which means that $HM(\mathrm{H}_n^{(I_a)})=\mathrm{H}_n^{(I_a)}$. 
Therefore, we verified the assumptions of Proposition~\ref{prop:fullfiberproduct} and so we conclude that $H=\mathrm{H}_n^{(I_a)}$, as needed.

For general $J$, if $\pi_{a'}:\prod_{a\in J}\prod_{i\in I_a}\mathrm{H}_a\rightarrow\prod_{i\in I_{a'}}\mathrm{H}_n$ denotes the projection map,
then $\pi_a(H)=\mathrm{H}_n^{(I_a)}$ for all $a\in J$ by the previous case,
hence $H\leq \prod_{a\in J}\mathrm{H}_n^{(I_a)}=\mathrm{H}_n^{\mathcal{I}_h}$.
By Lemma~\ref{lem:Mproduct}, 
$$
 M(\mathrm{H}_n^{\mathcal{I}_h})=\prod_{a\in J}M(\mathrm{H}_n^{(I_a)})=\prod_{a\in J}\prod_{i\in I_a}V_{n-1}\rtimes A_n,
$$ 
and $M(\mathrm{H}_n^{\mathcal{I}_h})\cap H$ corresponds to the field extension 
$$
 \prod_{a\in J}K_a\cdot K(\sqrt{{\rm discr}(f_i)},{i=1,\dots,k}).
$$ 
By Lemma \ref{lem:squareindep}, the family 
$$
 (f_0(0)+A_0+a)_{a\in J}\cup({\rm discr}(f_i))_{i=1,\dots,k}
$$ 
is square-independent,
which means that $H M(\mathrm{H}_n^{\mathcal{I}_h})=\mathrm{H}_n^{\mathcal{I}_h}$.
Therefore, Proposition \ref{prop:fullproduct}, applied to the groups $(\mathrm{H}_n^{(I_a)})_{a\in J}$ 
and the subgroup $H\leq \mathrm{H}_n^{\mathcal{I}_h}=\prod_{a\in J}\mathrm{H}_n^{(I_a)}$, gives that $H=\mathrm{H}_n^{\mathcal{I}_h}$.
\end{proof}






\subsection{Correlations of arithmetic functions depending on signed factorization type}
\label{sec:proofgeneral}

We are now ready to prove our general result on correlations in short intervals:

\begin{theorem}\label{thm:general}
Fix $k\geq 1$, $\psi_1,\dots,\psi_k\in\Lambda^*$, $1\geq\epsilon>0$ and $n>2\epsilon^{-1}$.
Then for $q$ an odd prime power, 
$f_0\in\mathbb{F}_q[T]$ monic of degree $n$ and $h_1,\dots,h_k\in\mathbb{F}_q[T]$ of degree less than $n$ and pairwise distinct,
\begin{eqnarray*}
  \left< \prod_{i=1}^k \psi_{i,q}(f+h_i) \right>_{|f-f_0|<|f_0|^\epsilon} &=&  
 \left<\psi_1,\dots,\psi_k\right>_{\mathrm{H}_n^{\mathcal{I}_h}}
 + O_{n,k,\psi}(q^{-1/2}) 
\end{eqnarray*}
where the implied constant depends only on $n$, $k$ and $\psi_1,\dots,\psi_k$,
and $\mathcal{I}_h$ is defined as in $(\ref{eqn:def_I_h})$.
\end{theorem}

\begin{proof}
The short interval $|f-f_0|<|f_0|^\epsilon$ is precisely the set of specializations
of $f_A=f_0+\sum_{i=0}^{m-1}A_iT^i$, with $m=\lceil\epsilon n\rceil$.
Setting $f_{A,i} = f_A+h_i$ and $f_{A,h}=\prod_{i=1}^kf_{A,i}$,
Proposition~\ref{prop:Galois_group_short_interval} gives that 
$$
 {\rm Gal}(f_{A,h}(-T^2)|\mathbb{F}_q(A))={\rm Gal}(f_{A,h}(-T^2)|\overline{\mathbb{F}}_q(A))=\mathrm{H}_n^{\mathcal{I}_h}.
$$
In particular, the $f_{A,i}$ are monic, square-free, not divisible by $T$ and pairwise coprime.
Therefore, the claim follows from Proposition \ref{prop:groupcorrelation}.
\end{proof}

We note that in the case where the $\psi_i$ depend only on {\em cycle type}, i.e.~$\psi_i(d,1,1)=\psi_i(d,1,-1)$ for all $d$,
$$
 \left<\psi_1,\dots,\psi_k\right>_{\mathrm{H}_n^{\mathcal{I}_h}}
 = \prod_{i=1}^k\left<{\psi}_{i}\right>_{\mathrm{H}_n},
$$
as we will explain in more detail in Lemma \ref{lem:cycle_type}, and hence
$$
 \left< \prod_{i=1}^k \psi_{i,q}(f+h_i) \right>_{f\in M_{n,q}}=
 \prod_{i=1}^k \left<\psi_{i,q}(f) \right>_{f\in M_{n,q}} + O(q^{-1/2}),
$$
which was proven already in \cite{ABR} (although not stated in such generality).

\section{Correlations in the large finite field limit}
\label{sec:applications}

We now prove Theorem \ref{thm:b} in short intervals and
compute the correlations of some further examples.
Due to Theorem \ref{thm:general} 
all that is left to do is to compute the corresponding
averages $\left<\psi_1,\dots,\psi_k\right>_{\mathrm{H}_n^{\mathcal{I}_h}}$,
which is a purely combinatorial task.

We will use on several occasions the (trivial) principles that if
$G_1$, $G_2$ are finite groups and $\psi_i\colon G_i\rightarrow\mathbb{C}$, then
\begin{eqnarray}
\label{eqn:psi1psi2}
\left<\psi_1(\sigma_1)\psi_2(\sigma_2)\right>_{(\sigma_1,\sigma_2)\in G_1\times G_2}
 &=&\left<\psi_1(\sigma)\right>_{\sigma\in G_1}\cdot \left<\psi_2(\sigma)\right>_{\sigma\in G_2}
\end{eqnarray}
and if $\pi\colon G_2\rightarrow G_1$ is an epimorphism, then 
\begin{eqnarray}
\label{eqn:psi_epi}
\left<\psi_1(\pi(\sigma))\right>_{\sigma\in G_2}&=&\left<\psi_1(\sigma)\right>_{\sigma\in G_1}.
\end{eqnarray}

\subsection{Autocorrelation of $b$ and $r$}
\label{sec:r}

We will start with the autocorrelation of $b$,
and then look at the closely related arithmetic function $r$, which counts the number of representations as a sum of two squares.
The following general consideration will simplify our arguments:

\begin{lemma}
\label{lem:constant_sign}
Let $\psi_1,\dots,\psi_k\in\Lambda^*$ 
and let $\mathcal{I}=(I_j)_{j\in J}$ be a partition of $\{1,\ldots, k\}$, i.e.~$I_j\neq\emptyset$ for all $j$ and $\{1,\dots,k\}=\coprod_{j\in J}I_j$.
If $\psi_1,\dots,\psi_k$ are all supported on $\chi^{-1}(1)=\{\lambda\in \Lambda : \chi(\lambda)=1\}$, then
$$
 \left<\psi_1,\dots,\psi_k\right>_{\mathrm{H}_n^{\mathcal{I}}} = \mathfrak{S}_\mathcal{I}\cdot  \left<\psi_{1}\right>_{\mathrm{H}_n}\cdots\left<\psi_{k}\right>_{\mathrm{H}_n},
$$
where $\mathfrak{S}_\mathcal{I}=2^{k-\#J}$.
\end{lemma}

\begin{proof}
Since by assumption $\psi_{i,n}(\sigma)=0$ for $\sigma\in\mathrm{H}_n\setminus V_{n-1}\rtimes S_n$ and since $\mathrm{H}_n^{\mathcal{I}}$ contains $\prod_i V_{n-1} \rtimes S_n$, we get by \eqref{eqn:psi1psi2} that
\[
\left<\psi_1,\dots,\psi_k\right>_{\mathrm{H}_n^{\mathcal{I}}} 
	=\frac{1}{\#\mathrm{H}_n^{\mathcal{I}}} \sum_{\sigma=(\sigma_i)\in  \prod_i V_{n-1}\rtimes S_n} \prod_{i} \psi_i(\sigma_i)=
	\frac{\#(V_{n-1}\rtimes S_n)^k}{\#\mathrm{H}_n^\mathcal{I}}\cdot \prod_{i=1}^k\left<\psi_{i}\right>_{V_{n-1}\rtimes S_n}.
\]
This finishes the proof as 
$\mathfrak{S}_{\mathcal{I}} = \frac{\#\mathrm{H}_n^k}{\#\mathrm{H}_n^\mathcal{I}}$. 
%
%
%
\end{proof}

The following result gives the autocorrelations of $b_q$ in short intervals. 
Theorem~\ref{thm:b} is the special case with parameters 
$\epsilon=1$ and $f_0=T^n$:

\begin{theorem}\label{thm:b_short}
Fix $k\geq 1$, $1\geq\epsilon>0$ and $n>2\epsilon^{-1}$.
Then for $q$ an odd prime power, 
$f_0\in\mathbb{F}_q[T]$ monic of degree $n$ and $h_1,\dots,h_k\in\mathbb{F}_q[T]$ of degree less than $n$ and pairwise distinct,
\begin{eqnarray*}
  \left< \prod_{i=1}^k b_q(f+h_i) \right>_{|f-f_0|<|f_0|^\epsilon} &=&
 \mathfrak{S}_{q,h} \cdot\left< b_q(f) \right>_{f\in M_{n,q}}^k + O_{n,k}(q^{-1/2}) \\
 &=&  \mathfrak{S}_h\cdot \frac{1}{4^{nk}}\left(2n\atop n\right)^k +O_{n,k}(q^{-1/2}) 
\end{eqnarray*}
where the implied constant depends only on $n$ and $k$, and
$\mathfrak{S}_{q,h}$ and $\mathfrak{S}_h$ are defined as in (\ref{eqn:def_Sqh}) and (\ref{eqn:def_Sh}).
\end{theorem}

\begin{proof}
Set $\psi_1=\dots=\psi_k=b$ and $\mathcal{I}=\mathcal{I}_h$
and note that Lemma \ref{lem:constant_sign} applies,
leading to 
$\left<b,\dots,b\right>_{\mathrm{H}_n^{\mathcal{I}_h}} = \mathfrak{S}_h\cdot \left< b\right>_{\mathrm{H}_n}^k$,
since $\mathfrak{S}_{\mathcal{I}_h}=\mathfrak{S}_h$.
Ewens' sampling formula gives that 
$\left<b\right>_{\mathrm{H}_n} = \frac{1}{4^n}\left(2n\atop n\right)$, 
see \cite[Lemma 2.3]{BBF}.
Therefore Theorem \ref{thm:general} shows that
\begin{eqnarray*}
  \left< \prod_{i=1}^k b_q(f+h_i) \right>_{|f-f_0|<|f_0|^\epsilon} 
  &=&  \mathfrak{S}_h\cdot \frac{1}{4^{nk}}\left(2n\atop n\right)^k +O_{n,k}(q^{-1/2}).
\end{eqnarray*}
In particular, $\left< b_q(f) \right>_{f\in M_{n,q}}=\frac{1}{4^n}\left(2n\atop n\right)+O(q^{-1/2})$
(which also follows from \cite{BSW}).
Together with $\mathfrak{S}_h=\mathfrak{S}_{q,h}+O(q^{-1/2})$ (Proposition \ref{prop:limSqh}),
we conclude that
\begin{eqnarray*}
    \mathfrak{S}_h\cdot \frac{1}{4^{nk}}\left(2n\atop n\right)^k
&=& \mathfrak{S}_{q,h} \cdot\left< b_q(f) \right>_{f\in M_{n,q}}^k + O_{n,k}(q^{-1/2}),
\end{eqnarray*}
as needed. 
\end{proof}

Just like the autocorrelation of $b$ (see the introduction), the autocorrelation of $r$ has been studied by various authors,
but the latter one turns out to be much more accessible:
Already Estermann \cite{Estermann} proves an asymptotic formula for $\left<r(n)r(n+h)\right>_{n\leq x}$.
Apparently unaware of that, Connors and Keating \cite[eqn.~(27)]{ConnorsKeating} provide 
a conjectural formula and numerics for 
$\lim_{x\rightarrow\infty}\left<r(n)r(n+h)\right>_{n\leq x}$
and observe that 
here the data seems to match their prediction even better than
in the case of $b$.
We define the function field analogue of $r$ as
\begin{eqnarray*}
 r_q(f) &=& \#\{(A,B) : f=A^2+TB^2, A,B\in\mathbb{F}_q[T]\}/\{\pm1\}.
\end{eqnarray*}
We note that $r_q$ is multiplicative and therefore one obtains, just like for $r$, the formula
\begin{eqnarray*}
 r_q(P_1^{e_1}\cdots P_r^{e_r}) &=& \begin{cases}\prod_{\chi_q(P_i)=1}(e_i+1),&\mbox{if }e_i\mbox{ is even for all $i$ with }\chi_q(P_i)=-1\\0,&\mbox{otherwise} \end{cases},
\end{eqnarray*}
which implies that $r_q$ is induced from $r\in \Lambda^*$ defined by 
\begin{equation}\label{eq:rdef}
 r(\lambda)=\begin{cases} \prod_{d,e}(e+1)^{\lambda(d,e,1)}, &\mbox{if }\lambda(d,2e+1,-1)=0\mbox{ for all }d,e \\ 0, &\mbox{otherwise} \end{cases}.
\end{equation}
We remark without proof that, like for integers, also the formula
\begin{eqnarray*}
 r_q(f) &=& \sum_{d|f,d\in M_q}\chi_q(d)
\end{eqnarray*}
holds.

We start by computing the mean value of $r$ on $\mathrm{H}_n$.
\begin{lemma}
\label{lem:average_r}
$\left<r\right>_{\mathrm{H}_n} = 1$
\end{lemma}

\begin{proof}
As usual we write an element  $\sigma\in \mathrm{H}_n$ as $\sigma = x \tau$ with $\tau\in S_n$ and $x=(x_1,\ldots, x_n) \in V =\mathbb{F}_2^n$ and we recall the corresponding signed factorization type $\lambda_\sigma$ attached to $\sigma$ and given in \eqref{def:lambdasig}.
Note that $\lambda_\sigma(d,e,s)=0$ if $e>1$.
Thus, 
if $\sum_{d}\lambda_\sigma(d,1,-1)=0$, then $\sum_{d,e}\lambda_\sigma(d,2e+1,-1)=0$,
$\prod_{d,e}(e+1)^{\lambda_\sigma(d,e,1)}=2^{\sum_d\lambda_\sigma(d,1,1)}$,
and $\sum_{d}\lambda_\sigma(d,1,1)=\sum_{d,s}\lambda_\sigma(d,1,s)=\sum_{d}\lambda_\tau(d,1,1)$.
By \eqref{eq:rdef} we therefore get
%
%
\begin{equation}
\label{eq:rcal} 
r_n(\sigma) = 
\begin{cases}
2^{\sum_d \lambda_\tau(d,1,1)}, & \mbox{if } \sum_d \lambda_\sigma(d,1,-1) =0\mbox{ and }\sigma=x\tau\\
0, & \mbox{otherwise.} 
\end{cases}
\end{equation}
Let $N_\tau$ denote the number of $\sigma=x\tau\in \mathrm{H}_n$ with $\sum_{d} \lambda_\sigma(d,1,-1)=0$. Note that $\sigma $ is counted by $N_\tau$ if and only if for each orbit $I\subset\{1,\ldots, n\}$ of $\tau$ we have $\sum_{i\in I}x_i=0$, and so there are $2^{n-j}$ such $\sigma$, where
\[
j= \sum_{d,s} \lambda_\tau(d,1,s) = \sum_{d} \lambda_\tau(d,1,1)
\] 
is the number of orbits of $\tau$.
So 
\begin{equation}
\label{cal:Ntau}
N_{\tau} = 2^{n- \sum_{d} \lambda_\tau(d,1,1)}.
\end{equation}
By \eqref{eq:rcal} and \eqref{cal:Ntau}, we conclude that 
\[
\sum_{\sigma \in \mathrm{H}_n} r_n(\sigma) = \sum_{\tau\in S_n}\sum_{x\in\mathbb{F}_2^n}r_n(x\tau)
=  \sum_{\tau\in S_n}N_\tau r_n(\tau) 
 = \sum_{\tau\in S_n} 2^n = n! 2^n = \#\mathrm{H}_n,
\]
and so $\left< r\right>_{\mathrm{H}_n} = 1$.
%
%
\end{proof}

Now we can compute the autocorrelation of $r$:

\begin{theorem}\label{thm:r}
Fix $k\geq 1$, $1\geq\epsilon>0$ and $n>2\epsilon^{-1}$.
Then for $q$ an odd prime power, 
$f_0\in\mathbb{F}_q[T]$ monic of degree $n$ and $h_1,\dots,h_k\in\mathbb{F}_q[T]$ of degree less than $n$ and pairwise distinct,
\begin{eqnarray*}
  \left< \prod_{i=1}^k r_q(f+h_i) \right>_{|f-f_0|<|f_0|^\epsilon} &=&
 \mathfrak{S}_{q,h} \cdot\left< r_q(f) \right>_{f\in M_{n,q}}^k + O_{n,k}(q^{-1/2}) \\
 &=&  \mathfrak{S}_h  +O_{n,k}(q^{-1/2}) 
\end{eqnarray*}
where the implied constant depends only on $n$ and $k$, and
$\mathfrak{S}_{q,h}$ and $\mathfrak{S}_h$ are defined as in (\ref{eqn:def_Sqh}) and (\ref{eqn:def_Sh}).
\end{theorem}

\begin{proof}
Apply Theorem \ref{thm:general} with $\psi_1=\dots=\psi_k=r$.
Note that Lemma \ref{lem:constant_sign} applies again,
so the claim follow from Lemma \ref{lem:average_r}.
\end{proof}

\subsection{Cross-correlations of $b$ and $r$ with $1_\mathbb{P}$}

We now turn to cross-correlations of $b$ and $r$ with 
the prime indicator function $1_\mathbb{P}$,
which also have been studied by various authors:
Starting from a conjecture of Hardy and Littlewood \cite[Conjecture J]{HL}
on the number of representations of an integer as the sum of two squares and a prime,
\cite[Theorem 2]{Hooley0} proves an asymptotic formula for 
$$
 \frac{1}{x}\sum_{p\leq x\mbox{\;\;\scriptsize prime}}r(p+h)=\left<r(n)1_{\mathbb{P}}(n-h)\right>_{n\leq x}
$$ 
under the Extended Riemann Hypothesis,
which \cite{Bredihin} then proves unconditionally.
Motohashi in \cite[Conjecture J*, Theorem 2]{Motohashi} and \cite{Motohashi2}
gives a conjectural asymptotic formula for 
$$
 \frac{1}{x}\sum_{p\leq x\mbox{\;\;\scriptsize prime}}b(p-1)=\left<b(n)1_{\mathbb{P}}(n+1)\right>_{n\leq x-1}
$$ 
and proves upper and lower bounds. 
Iwaniec \cite[p.~204]{Iwaniec1972} also proves lower and upper bounds and suggests a correction of Motohashi's conjecture. 
We now give a function field version of the Motohashi--Iwaniec conjecture 
(or rather a generalization of it from $h=1$ to arbitrary $h$)
and prove it in the large finite field limit.

The heuristics for the cross-correlation of $b$ and $1_{\mathbb{P}}$ is very similar to that of the auto-correlation of $b$ discussed in Section~\ref{sec:heuristics}, hence we omit some of the details, and leave them as an exercise for the reader.
Let $h\in\mathbb{F}_q[T]$ be non-zero. For each prime polynomial $P$, we let 
\begin{align}
\rho_{q,h}(P) &= \lim_{\nu \to \infty} \frac{ \#\{ f \mbox{ mod } P^{\nu} : 
f\in M_q\mbox{ prime and }\exists A,B: f+h\equiv A^2+TB^2\mbox{ mod }P^\nu\}}{|P|^\nu},\\
\rho_q (P) &= \lim_{\nu \to \infty} \frac{\#\{ (f,g)\mbox{ mod } P^{\nu} : f\in M_q \mbox{ prime and } \exists A,B: g \equiv A^2+TB^2 \mbox{ mod } P^{\nu}\}}{|P|^{2\nu}}.
\end{align}
One may verify that the limits indeed exist. So $\rho_{q,h}(P)/\rho_q(P)$ measures the local deviation at $P$ from the random model. 
One then may make the following 
\begin{conjecture}\label{conjecture:bLambda}
Fix $N\geq 1$. Then for $q$ an odd prime power, $n\geq N$, and nonzero $h\in \FF_q[T]$ of degree less than $N$, 
\begin{eqnarray*}
\left< 1_\mathbb{P}(f)b_q(f+h)\right>_{f\in M_{n,q}} &\sim& \mathfrak{T}_{q,h}\cdot \left<1_\mathbb{P}(f)\right>_{f\in M_{n,q}}\cdot\left<b_q(f) \right>_{f\in M_{n,q}}\\
&\sim&\mathfrak{T}_{q,h}\cdot K_q\cdot \frac{1}{n 4^n}\left(2n\atop n\right),
\end{eqnarray*}
uniformly as $q^n \to \infty$, where $K_q$ is defined as in \eqref{eqn:def_Kq} and 
\begin{equation}\label{conj:lambdab}
\mathfrak{T}_{q,h} = \prod_{\substack{P\in\mathbb{F}_q[T]\\\mathrm{monic\; irred.}}} \frac{\rho_{q,h}(P)}{\rho_q(P)}.
\end{equation}
\end{conjecture}
We note that $\mathfrak{T}_{q,h}$ converges. Next we give formulas for the $\rho$'s.
Note that $f$ is congruent to a prime modulo $P^\nu$
if and only if either $f\not\equiv 0\mbox{ mod }P$ or $f\equiv P\mbox{ mod }P^\nu$,
and therefore replacing the condition ``$f$ prime'' by ``$f\not\equiv 0\mbox{ mod }P$''
leads to the same limits and at the same time simplifies the computations.

\begin{lemma}\label{lem:rho}
Let $P\in\mathbb{F}_q[T]$ monic irreducible.
\begin{enumerate}
\item $\rho_q(P) = (1-|P|^{-1})\delta_{q,0}(P)$ (Recall that $\delta_{q,0}(P)$ is explicitly given by Corollary~\ref{lem:delta0}.)
\item If $\chi(P)=1$, then $\rho_{q,h}(P) = \rho_q(P)= 1-|P|^{-1}$. 
\item If $\chi(P)=-1$, then 
\[
\rho_{q,h}(P) = 
\begin{cases}
1-|P|^{-1}, & \mbox{if $P\mid h$}\\
1-|P|^{-1}-(1+|P|)^{-1}, &\mbox{otherwise.}
\end{cases}
\]
\item If $\chi(P)=  0$, then 
\[
\rho_{q,h}(P) = \frac{1}{2} - \frac{1}{2q} (1+\chi_q(h)).
\]
\end{enumerate}
\end{lemma}

\begin{proof}
(1) is immediate and (2) follows directly from Lemma~\ref{lem:characterize}\eqref{lem:characterize3} while 
(3) and (4) follows from Lemma~\ref{lem:characterize} using similar arguments as those used in the proof of Proposition~\ref{prop:Sfor0h}.
\end{proof}

Just like we deduced Proposition \ref{prop:limSqh} from Proposition \ref{lem:delta_chi_-1} and Corollary \ref{lem:delta0},
we can use Lemma~\ref{lem:rho} to conclude that
\[
\mathfrak{T}_{q,h} = 1+O(q^{-1})
\]
where the implied constant depends only on the degree of $h$. 
Since also $K_q=1+O(q^{-1})$, the next results proves Conjecture~\ref{conjecture:bLambda} in the large finite field limit.

\begin{theorem}
\label{thm:r_Lambda}
Fix $1\geq\epsilon>0$ and $n>2/\epsilon$.
Then for $q$ an odd prime power, 
$f_0\in\mathbb{F}_q[T]$ monic of degree $n$ and $h\in\mathbb{F}_q[T]$ of degree less than $n$,
\begin{eqnarray*}
  \left< 1_\mathbb{P}(f)b_q(f+h) \right>_{|f-f_0|<|f_0|^\epsilon} &=&
  \mathfrak{T}_{q,h}\cdot\left< 1_\mathbb{P}(f) \right>_{f\in M_{n,q}}\cdot\left< b_q(f) \right>_{f\in M_{n,q}} + O_{n}(q^{-1/2}) \\
 &=&  \frac{1}{n4^{n}}\left(2n\atop n\right) +O_{n}(q^{-1/2}) 
\end{eqnarray*}
where $\mathfrak{T}_{q,h}$ is defined as in \eqref{conj:lambdab} and the implied constant depends only on $n$.
\end{theorem}

\begin{proof}
Apply Theorem~\ref{thm:general} with $k=2$, $h_1=0$, $h_2=-h$, $\psi_1=b$, $\psi_2=1_\mathbb{P}$.
Note that $\psi_2$ satisfies the assumption of the following lemma,
which proves the first equality.
For the second equality use that 
$\mathfrak{T}_{q,h}=1+O(q^{-1})$ and
$\left<b\right>_{\mathrm{H}_n}=\frac{1}{4^n}\left(2n\atop n\right)$ (see above),
and that $\left<1_\mathbb{P}\right>_{\mathrm{H}_n}=\frac{1}{n}$, as this is the fraction of $n$-cycles in $S_n$.
\end{proof}

\begin{lemma}
\label{lem:cycle_type}
Let $\psi_1,\dots,\psi_k\in\Lambda^*$ and 
let $\mathcal{I}=(I_j)_{j\in J}$ be a partition of $\{1,\ldots, k\}$ as above.
If $\psi_k$ depends only on cycle type in the sense that $\psi_k(d,1,1)=\psi_k(d,1,-1)$ for all $d$, then
$$
 \left<\psi_1,\dots,\psi_k\right>_{\mathrm{H}_n^{\mathcal{I}}}
  = \left<\psi_1,\dots,\psi_{k-1}\right>_{\mathrm{H}_n^{\mathcal{I}'}}\cdot\left<\psi_k\right>_{\mathrm{H}_n},
$$
where $\mathcal{I}'=(I_j')_{j\in J'}$, $I_j'=I_j\setminus\{k\}$, $J'=\{j\in J:I_j'\neq\emptyset\}$.
\end{lemma}

\begin{proof}
Let $\pi\colon \mathrm{H}_n\rightarrow S_n$ be the quotient map.
The assumption implies that $\psi_{k,n}=\tilde{\psi}_{k,n}\circ\pi$
with a function $\tilde{\psi}_{k,n}\colon S_n\rightarrow\mathbb{C}$.
Observe that restricting the homomorphism 
$$
 {\rm id}_{\mathrm{H}_n^{\mathcal{I}'}}\times\pi\colon \mathrm{H}_n^{\mathcal{I}'}\times \mathrm{H}_n\rightarrow\mathrm{H}_n^{\mathcal{I}'}\times S_n
$$
gives an epimorphism $\mathrm{H}_n^\mathcal{I}\rightarrow \mathrm{H}_n^{\mathcal{I}'}\times S_n$.
Thus,
\begin{eqnarray*}
 \left<\psi_1,\dots,\psi_k\right>_{\mathrm{H}_n^{\mathcal{I}}}
 &\stackrel{(\ref{eqn:psi_epi})}{=}&\left<\prod_{i=1}^{k-1}\psi_{i,n}(\pi_i(\sigma))\cdot\tilde{\psi}_{k,n}(\pi_k(\sigma))\right>_{\sigma\in\mathrm{H}_n^{\mathcal{I}'}\times S_n}\\
 &\stackrel{(\ref{eqn:psi1psi2})}{=}&\left<\prod_{i=1}^{k-1}\psi_{i,n}(\pi_i(\sigma))\right>_{\sigma\in\mathrm{H}_n^{\mathcal{I}'}}\cdot\left<\tilde{\psi}_{k,n}(\sigma)\right>_{\sigma\in S_n}\\
&\stackrel{(\ref{eqn:psi_epi})}{=}& \left<\psi_1,\dots,\psi_{k-1}\right>_{\mathrm{H}_n^{\mathcal{I}'}}\cdot\left<\psi_k\right>_{\mathrm{H}_n}.
\end{eqnarray*}
\end{proof}

Note that although in Theorem \ref{thm:r_Lambda} the functions $b_q$ and $1_{\mathbb{P}}$ become independent in the large finite field limit,
this could not have been deduced from the earlier results in \cite{ABR} (see the remark after Theorem \ref{thm:general}),
as only one of the two arithmetic functions, namely $1_\mathbb{P}$, depends only on cycle type, 
while the other one does not.


\subsection{Autocorrelations of $d_r\chi$}
\label{sec:dkchi}

Let $d_r(n)$ be the number of ways to write $n$ as a product of $r$ positive integers. 
In particular, $d_2=\tau$ is the usual divisor function. 
The problem of estimating the autocorrelations of $d_r$, sometimes referred to as `additive divisor problem', `shifted divisor', or `shifted convolution', is well studied both in number fields 
(see e.g.~\cite{ConreyGonek,KeatingGonekHughes,ConreyKeatingIV} and the recent survey \cite{BlogTao})
and function fields (see \cite{ABR}).
The asymptotic of the cross-correlations of the divisor functions are related to computing the moments of the zeta function on the critical line, see \cite{Ivic}. 
We consider a twisted version of this; namely, we twist $d_r$ by a quadratic character $\chi$ and study the cross-correlations of the $d_r \chi$. 
This is closely related to the moments of the corresponding $L$-function $L(s,\chi)$ on the critical line, but we do not elaborate on this any further, since the goal of this section is to provide yet another application of Theorem~\ref{thm:general}.

In the function field setting, we let $d_{r,q}(f)$ denote the number of ways to write $f$ as a product of $r$ monic polynomials,
and we twist $d_{r,q}$ by $\chi_q$, the quadratic character modulo $T$.
We note that in this setting,
computing the moments of the corresponding L-function $L_q(s,\chi_q)$ is trivial, 
since $L_q(s,\chi)$ is identically equal to $1$,
but the cross-correlations of the $d_{r,q}\chi_q$ are nevertheless interesting
and, to the best of our knowledge, unknown.
We now compute these cross-correlations in the large finite field limit,
which we can do as $d_{r,q}$ is induced from $d_r\in\Lambda^*$ given by
$$
 d_r(\lambda)=\prod_{d,e,s}\left(e+r-1\atop r-1\right)^{\lambda(d,e,s)}.
$$

\begin{lemma}
\label{lem:dr_average}
For every $r$ and $n$,
\begin{eqnarray*}
 \left<d_r\right>_{\mathrm{H}_n} &=& \left(n+r-1\atop r-1 \right).
\end{eqnarray*}
\end{lemma}

\begin{proof}
Since $d_{r,n}$ factors through $\pi\colon \mathrm{H}_n\rightarrow S_n$, we have
\begin{eqnarray*}
 \left<d_r\right>_{\mathrm{H}_n} 
 &=&\frac{1}{2^nn!}\sum_{\sigma\in\mathrm{H}_n} \left(1+r-1\atop r-1\right)^{\sum_{d,s}\lambda_\sigma(d,1,s)}
 \;=\; \frac{1}{n!}\sum_{\tau\in S_n}r^{\omega(\tau)}
\end{eqnarray*}
where $\omega(\tau)$ is the number of cycles of $\tau$. 
Viewing $r^{\omega(\tau)}$ as the number of partitions of $\{1,\dots,n\}$ into $r$ sets that 
are unions of orbits of $\tau$ and changing order of summation we get that
\begin{eqnarray*}
\sum_{\tau\in S_n}r^{\omega(\tau)} &=& \sum_{\{1,\dots,n\}=\coprod_{i=1}^rA_i}\prod_{i=1}^r\#A_i!
\end{eqnarray*}
Splitting the sum according to the cardinality of $A_1$ and applying induction on $r$ we conclude
\begin{eqnarray*}
 \sum_{\{1,\dots,n\}=\coprod_{i=1}^rA_i}\prod_{i=1}^r\#A_i! &=& \sum_{a=0}^n\left(n\atop a\right)a!(n-a)!\left(n-a+r-2\atop r-2\right)\\
  &=& n!\sum_{a=0}^n\left(a+r-2\atop r-2\right) \;=\; n!\left(n+r-1\atop r-1\right).
\end{eqnarray*}
\end{proof}

\begin{lemma}
\label{lem:psichi_average}
Let $\psi\in\Lambda^*$ with $\psi(d,1,1)=\psi(d,1,-1)$ for all $d$.
Then $\left<\psi\chi\right>_{\mathrm{H}_n}=0$.
\end{lemma}

\begin{proof}
Write $\psi_n=\tilde{\psi}_n\circ\pi$ with $\pi\colon \mathrm{H}_n\rightarrow S_n$ the quotient map. Then
$$
 \sum_{\sigma\in\mathrm{H}_n}\psi_n(\sigma)\chi_n(\sigma) = \sum_{\tau\in S_n}\sum_{\sigma\in\pi^{-1}(\tau)}\psi_n(\sigma)\chi_n(\sigma)
 =\sum_{\tau\in S_n}\tilde{\psi}_n(\tau)\sum_{\sigma\in\pi^{-1}(\tau)}\chi_n(\sigma)=0,
$$
since $\chi_n(\sigma)=1$ for half of the $2^n$ many $\sigma$ in each $\pi^{-1}(\tau)$,
and $\chi_n(\sigma)=-1$ for the other half.
\end{proof}

\begin{lemma}
\label{lem:twisted}
Let $\psi_1,\dots,\psi_k\in\Lambda^*$ and let $\mathcal{I}=(I_j)_{j\in J}$ be a partition of $\{1,\dots,k\}$.
If $\psi_i(d,1,1)=\psi_i(d,1,-1)$ for all $i$ and $d$,
then
\begin{eqnarray*}
 \left<\psi_1\chi,\dots,\psi_k\chi\right>_{\mathrm{H}_n^\mathcal{I}}&=& 
 \begin{cases}\left<\psi_1\right>_{\mathrm{H}_n}\cdots\left<\psi_k\right>_{\mathrm{H}_n},&\mbox{if }\#I_j\mbox{ is even for all }j\in J\\0,&\mbox{otherwise}\end{cases}.
\end{eqnarray*}
\end{lemma}

\begin{proof}
First observe that since $\mathrm{H}_n^\mathcal{I}=\prod_{j\in J}\mathrm{H}_n^{(I_j)}$,
by principle (\ref{eqn:psi1psi2}) it suffices to prove the claim in the case $\#J=1$,
which we therefore assume now.
Note that $\left<\psi_i\chi\right>_{\mathrm{H}_n}=0$ by Lemma \ref{lem:psichi_average}
and thus $\left<\psi_i(\chi+1)\right>_{\mathrm{H}_n}=\left<\psi_i\right>_{\mathrm{H}_n}$.
So since $\psi_i(\chi+1)$ satisfies the assumptions of Lemma \ref{lem:constant_sign}
and $\psi_i$ satisfies the assumptions of Lemma \ref{lem:cycle_type}, we get
\begin{eqnarray*}
\left<\psi_1\chi,\dots,\psi_k\chi\right>_{\mathrm{H}_n^\mathcal{I}} 
&=& \left<\psi_1(\chi+1)-\psi_1,\dots,\psi_k(\chi+1)-\psi_k\right>_{\mathrm{H}_n^\mathcal{I}}\\
&=& \prod_{i=1}^k\left<-\psi_i\right>_{\mathrm{H}_n}+\sum_{\emptyset\neq S\subseteq\{1,\dots,k\}}\left(\prod_{i\notin S}\left<-\psi_i\right>_{\mathrm{H}_n}\cdot 2^{\#S-1}\prod_{i\in S}\left<\psi_i(\chi+1)\right>_{\mathrm{H}_n}\right)\\
&=& (-1)^k\cdot\prod_{i=1}^k\left<\psi_i\right>_{\mathrm{H}_n}\cdot\left(1+\sum_{\emptyset\neq S\subseteq\{1,\dots,k\}}(-1)^{\#S}2^{\#S-1}\right).
\end{eqnarray*}
Now note that 
\begin{eqnarray*}
 \sum_{S\subseteq\{1,\dots,k\}}(-1)^{\#S}2^{\#S} 
&=&  \sum_{ S\subseteq\{1,\dots,k\}}\sum_{S_0\subseteq S}(-1)^{\#S}
=  \sum_{S_0\subseteq\{1,\dots,k\}} \sum_{S_0\subseteq S\subseteq\{1,\dots,k\}}(-1)^{\#S}
\end{eqnarray*} 
and $\sum_{S_0\subseteq S\subseteq\{1,\dots,k\}}(-1)^{\#S}$ equals $0$ except if $S_0=\{1,\dots,k\}$,
in which case it equals $(-1)^k$. Thus,
$$
 \left<\psi_1\chi,\dots,\psi_k\chi\right>_{\mathrm{H}_n^\mathcal{I}}  
= (-1)^k\cdot\prod_{i=1}^k\left<\psi_i\right>_{\mathrm{H}_n}\cdot (1+\frac{1}{2}((-1)^k-1)),
$$
from which the claim follows.
\end{proof}

\begin{theorem}
\label{thm:dkchi}
Fix $k\geq 1$, $r_1,\dots,r_k\geq2$, $1\geq\epsilon>0$ and $n>2\epsilon^{-1}$.
Then for $q$ an odd prime power, 
$f_0\in\mathbb{F}_q[T]$ monic of degree $n$ and $h_1,\dots,h_k\in\mathbb{F}_q[T]$ of degree less than $n$ and pairwise distinct,
\begin{eqnarray*}
  \left< \prod_{i=1}^k d_{r_i,q}\chi_q(f+h_i) \right>_{|f-f_0|<|f_0|^\epsilon} 
  &=&\mathfrak{D}_h\cdot\prod_{i=1}^k\left<d_{r_i,q}(f)\right>_{f\in M_{n,q}} + O(q^{-1/2})\\
 &=&\mathfrak{D}_h\cdot\prod_{i=1}^k\left(n+r_i-1\atop r_i-1\right) + O(q^{-1/2})
\end{eqnarray*}
where the implied constant depends only on $n$ and $r_1,\dots,r_k$, and
$$
 \mathfrak{D}_h = \begin{cases}1,&\mbox{if }\#\{i:h_i(0)=a\}\mbox{ is even for all }a\in\mathbb{F}_q \\ 0,&\mbox{otherwise}\end{cases}
$$
\end{theorem}

\begin{proof}
Apply Theorem \ref{thm:general} with $\psi_i=d_{r_i}\chi$.
Note that $d_{r_1},\dots,d_{r_k}$ satisfy the assumptions of Lemma \ref{lem:twisted},
and their averages are given by Lemma \ref{lem:dr_average}.
\end{proof}

\section*{Acknowledgements}

The authors are very grateful to Ofir Gorodetsky for sharing his signed factorization type viewpoint on sums of two squares with them.
They would also like to thank
Efrat Bank for helpful discussions on the topic of this work, 
Hung Bui for suggesting to study correlations relating to moments of $L$-functions,
Jon Keating and Zeev Rudnick for their advice on a preliminary version of the paper, 
and 
Tristan Freiberg, P\"ar Kurlberg and Lior Rosenzweig
for making their preliminary manuscript available to them.

The first author is partially supported by the Israel Science Foundation (grant No. 925/14),
the second author by a research grant from the Ministerium f\"ur Wissenschaft, Forschung und Kunst Baden-W\"urttemberg.


\begin{thebibliography}{BSW16}



\bibitem[ABR15]{ABR}
J.\ C.~Andrade, Lior Bary-Soroker, and Zeev Rudnick.
\newblock Shifted convolution and the Titchmarsh divisor problem over $\mathbb{F}_q[t]$.
\newblock {\em Philosophical Transactions of the Royal Society of London A: Mathematical, Physical and Engineering Sciences},
Theo Murphy meeting issue `Number fields and function fields: coalescences, contrasts and emerging applications' compiled and edited by J. P. Keating, Z. Rudnick and T. D. Wooley, 373(2040), 2015.


\bibitem[BB15]{BB}
E.~Bank and L.~Bary-Soroker.
\newblock Prime polynomial values of linear functions in short intervals.
\newblock {\em J. Number Theory} 151:263--275, 2015.

\bibitem[BBF17]{BBF}
E.~Bank, L.~Bary-Soroker and A.~Fehm.
\newblock Sums of two squares in short intervals in polynomial rings over finite fields.
\newblock To appear in {\em American J.\ Math.}, 2017. 

\bibitem[Ban86]{Bantle}
G.~Bantle.
\newblock An asymptotic formula for $B$-twins.
\newblock {\em Acta Arithmetica} 47:297--312, 1986.



\bibitem[BSW16]{BSW}
L.~Bary-Soroker, Y.~Smilansky, and A.~Wolf.
\newblock On the function field analogue of Landau's theorem on sums of squares.
\newblock {\em Finite Fields Appl.} 39:195--215, 2016

\bibitem[Bre62]{Bredihin}
B.\ M.\ Bredihin.
\newblock Binary additive problems with prime numbers. 
\newblock {\em Dokl. Akad. Nauk SSSR} 142:766--768, 1962.

\bibitem[Car15]{Carmon}
D.\ Carmon.
\newblock The autocorrelation of the M\"obius function and Chowla's conjecture for the rational function field in characteristic $2$.
\newblock {\em Philosophical Transactions of the Royal Society of London A: Mathematical, Physical and Engineering Sciences},
Theo Murphy meeting issue `Number fields and function fields: coalescences, contrasts and emerging applications' compiled and edited by J. P. Keating, Z. Rudnick and T. D. Wooley, 373(2040), 2015.

\bibitem[CR14]{CarmonRudnick}
D.\ Carmon and Z.\ Rudnick.
\newblock The autocorrelation of the M\"obius function and Chowla's conjecture for the rational function field.
\newblock  {\em Q. J. Math.} 65(1):53--61, 2014.

\bibitem[CD87]{CochraneDressler}
T.~Cochrane and R.~E.~Dressler.
\newblock Consecutive triples of sums of two squares.
\newblock {\em Archiv Math.} 49:301--304, 1987.

\bibitem[CK97]{ConnorsKeating}
R.~D.~Connors and J.~P.~Keating.
\newblock Two-point spectral correlations for the square billiard
\newblock {\em J. Phys. A: Math. Gen.} 30:1817--1830, 1997.

\bibitem[CG01]{ConreyGonek}
J.~B.~Conrey and S.~M.~Gonek.
\newblock High moments of the Riemann zeta-function. 
\newblock {\em Duke Math. J.} 107(3):577--604, 2001.




\bibitem[CK16]{ConreyKeatingIV}
B.~Conrey and J.~P.~Keating.
\newblock Moments of zeta and correlations of divisor-sums: IV.
\newblock {\em Res. Number Theory} 2, 2016.

\bibitem[Est32]{Estermann}
T.~Estermann.
\newblock An asymptotic formula in the theory of numbers.
\newblock {\em Proc. London Math. Soc.} 34:280--292, 1932.


\bibitem[FJ08]{FJ}
M.~D.~Fried and M.~Jarden.
\newblock {\em Field Arithmetic}.
\newblock Third Edition. Springer, 2008.



\bibitem[FKR17]{FKR}
T.~Freiberg, P.~Kurlberg and L.~Rosenzweig.
\newblock Poisson distribution for gaps between sums of two squares and level spacings for toral point scatterers.
\newblock arXiv:1701.01157 [math-ph], 2017.


\bibitem[Gor16]{Gorodetsky}
O.~Gorodetsky.
\newblock A Polynomial Analogue of Landau's Theorem and Related Problems.
\newblock  	arXiv:1603.02890 [math.NT], 2016.

\bibitem[Hal06]{Hall}
C.~Hall.
\newblock $L$-functions of twisted Legendre curves.
\newblock  {\em J. Number Theory} 119(1):128--147, 2006.

\bibitem[HL24]{HL}
G.~H.~Hardy and J.~E.~Littlewood.
\newblock Some Problems of 'Partitio Numerorum'(V): A Further Contribution to the Study of Goldbach's Problem.
\newblock {\em Proc. London Math. Soc.} S2-22 no. 1, 46--56, 1924.


\bibitem[Hoo57]{Hooley0}
C.~Hooley.
\newblock On the representation of a number as the sum of two squares and a prime.
\newblock {\em Acta Math.} 97:189--210, 1957.

\bibitem[Hoo71]{Hooley1}
C.~Hooley.
\newblock On the intervals between numbers that are sums of two squares.
\newblock {\em Acta Math.} 127(1):279--297, 1971.

\bibitem[Hoo73]{Hooley2}
C.~Hooley.
\newblock On the intervals between numbers that are sums of two squares. II.
\newblock {\em J.\ Number Theory} 5:215--217, 1973.

\bibitem[Hoo74]{Hooley3}
C.~Hooley.
\newblock On the intervals between numbers that are sums of two squares. III.
\newblock {\em J. Reine Angew. Math.} 267(1):207--218, 1974.

\bibitem[Hoo94]{Hooley4}
C.~Hooley.
\newblock On the intervals between numbers that are sums of two squares. IV.
\newblock {\em J. Reine Angew. Math.} 452:79--109, 1994.

\bibitem[IS72]{IndlekoferSchwarz}
K.-H.~Indlekofer and W.~Schwarz.
\newblock \"Uber B-Zwillinge. 
\newblock {\em Archiv Math.} 23:251--256, 1972

\bibitem[Ind74]{Indlekofer}
K.-H.~Indlekofer.
\newblock Scharfe untere Absch\"atzung f\"ur die Anzahlfunktion der B-Zwillinge. 
\newblock {\em Acta Arithmetica} 26:207--212, 1974.

\bibitem[Ivi97]{Ivic}
A.~Ivic.
\newblock The general additive divisor problem and moments of the zeta-function.
\newblock {\em New trends in probability and statistics} 4:69--89, 1997.


\bibitem[Iwa72]{Iwaniec1972}
H.~Iwaniec.
\newblock Primes of the type $\varphi(x,y)+A$ where $\varphi$ is a quadratic form.
\newblock {\em Acta Arithmetica} 21:203--234, 1972.

\bibitem[Iwa76]{Iwaniec}
H.~Iwaniec.
\newblock The half dimensional sieve.
\newblock {\em Acta Arithmetica} 29:69--95, 1976.


\bibitem[KGH07]{KeatingGonekHughes}
J.~P.~Keating, S.~M.~Gonek, and C.~P.~Hughes.
\newblock A hybrid Euler-Hadamard product for the Riemann zeta function.
\newblock {\em Duke Math. J} 136:507--549, 2007.

\bibitem[KR16]{KRG}
J.~P.~Keating and E.~Roditty-Gershon.
\newblock Arithmetic correlations over large finite fields.
\newblock {\em Int. Math. Res. Not.} IMRN 2016, no. 3, 860--874, 2016.

\bibitem[KR14]{KeatingRudnick}
J.~P.~Keating and Z.~Rudnick.
\newblock The variance of the number of prime polynomials in short intervals and in residue classes.
\newblock {\em Int. Math. Res. Not.} IMRN 2014(1):259--288, 2014. 

\bibitem[Kel78]{Kelly}
P.~J.~Kelly.
\newblock The number of B-twins in an interval.
\newblock Dissertation, Nottingham, 1978.


\bibitem[Lan08]{Landau}
E.~Landau.
\newblock \"{U}ber die {E}inteilung der positiven ganzen {Z}ahlen in vier {K}lassen nach der {M}indestzahl der zu ihrer additiven {Z}usammensetzung erforderlichen {Q}uadrate.
\newblock {\em Arch. Math. Phys.} 13:305--312, 1908.



\bibitem[Mil14]{Milne}
J.~Milne.
\newblock {\em Fields and Galois theory}.
\newblock Lecture notes, version 4.50, 2014.

\bibitem[Mot70]{Motohashi}
Y.~Motohashi.
\newblock On the distribution of prime numbers which are of the form $x^2+y^2+1$.
\newblock {\em Acta Arithmetica} XVI:351--363, 1970.

\bibitem[Mot71]{Motohashi2}
Y.~Motohashi.
\newblock On the distribution of prime numbers which are of the form ``$x^2+y^2+1$''. II.
\newblock {\em Acta Math. Acad. Sci. Hungar.} 22:207--210, 1971.

\bibitem[Pol08]{Pollack}
P.~Pollack.
\newblock Simultaneous prime specializations of polynomials over finite fields.
\newblock {\em Proc. Amer. Math. Soc.} 136(11):3775--3784, 2008.

\bibitem[Poo03]{Poonen}
B.~Poonen.
\newblock Squarefree values of multivariable polynomials.
\newblock  {\em Duke Math. J.} 118(2):353--373, 2003.

\bibitem[Rie65]{Rieger}
G.~J.~Rieger.
\newblock Aufeinanderfolgende Zahlen als Summen von zwei Quadraten.
\newblock {\em Indag. Math.} 27:208--220, 1965.

\bibitem[Ros02]{Rosen}
M.~Rosen.
\newblock {\em Number Theory in Function Fields}.
\newblock Springer, 2002.

\bibitem[Rud14]{Rudnick}
Z.~Rudnick.
\newblock Some problems in analytic number theory for polynomials over a finite field. 
\newblock {\em Proceedings of the ICM vol 1}, 2014.

\bibitem[Sage]{Sage}
SageMath, the Sage Mathematics Software System (Version SageMath-7.2.beta0),
   The Sage Developers, 2016, \url{http://www.sagemath.org}.

\bibitem[Sch76]{Schmidt}
W.~M.~Schmidt.
\newblock {\em Equations over Finite Fields. An Elementary Approach.}
\newblock Springer 1976.

\bibitem[Sch72]{Schwarz}
W.~Schwarz.
\newblock \"Uber $B$-Zwillinge II. 
\newblock {\em Archiv Math.} 23:408--409, 1972.

\bibitem[Smi13]{Smilansky}
Y.~Smilansky.
\newblock Sums of two squares - pair correlation \& distribution in short intervals.
\newblock {\em Int. J. Number Theory} 09, 2013.

\bibitem[Tao16]{BlogTao}
T.~Tao.
\newblock Heuristic computation of correlations of higher order divisor functions.
\newblock {\em WordPress.com},  What's new, Online Blog. 
\url{http://goo.gl/GBncWw}



\end{thebibliography}
\end{document}